\documentclass[12pt,reqno]{amsart}

\usepackage{hyperref}

\usepackage[T1]{fontenc}
\usepackage[latin1]{inputenc}
\usepackage{latexsym}
\usepackage[arrow, matrix, curve]{xy}
\usepackage[dvips]{graphicx}
\usepackage{epsfig}
\usepackage{bm}
\usepackage{bbm} 
\usepackage{accents}
\usepackage{xcolor}
\usepackage{amssymb}
\usepackage{amsmath}
\usepackage{verbatim}
\usepackage{amsthm}
\usepackage{enumerate}
\usepackage{mathrsfs}
\usepackage[hmargin=3cm,vmargin=3cm]{geometry}
\DeclareFontFamily{U}{mathb}{\hyphenchar\font45}
\DeclareFontShape{U}{mathb}{m}{n}{
<-6> mathb5 <6-7> mathb6 <7-8> mathb7
<8-9> mathb8 <9-10> mathb9
<10-12> mathb10 <12-> mathb12
}{}
\DeclareSymbolFont{mathb}{U}{mathb}{m}{n}
\DeclareMathSymbol{\llcurly}{\mathrel}{mathb}{"CE}
\DeclareMathSymbol{\ggcurly}{\mathrel}{mathb}{"CF}

\usepackage[english]{babel}
\usepackage{mathtools}
\usepackage{todonotes}
\usepackage{url}
\usepackage[norefs,nocites]{refcheck}
\usepackage{hyperref}

\newtheorem{theorem}{Theorem}[section]
\newtheorem{lemma}[theorem]{Lemma}
\newtheorem{corollary}[theorem]{Corollary}

\newtheorem{proposition}[theorem]{Proposition}

\theoremstyle{definition}
\newtheorem*{definition}{Definition}

\theoremstyle{remark}

\numberwithin{equation}{section}

\newcommand{\mmod}[1]{\,\,(\text{mod}\,\,#1)}

\def\bb{{\mathbf b}}
\def\bc{{\mathbf c}}

\def\bh{{\mathbf h}}

\def\bx{{\mathbf x}}
\def\by{{\mathbf y}}

\def\calA{{\mathcal A}} 
\def\calB{{\mathcal B}} 
\def\calC{{\mathcal C}} 
\def\calD{{\mathcal D}}
\def\calE{{\mathcal E}}
\def\calF{{\mathcal F}}

\def\calI{{\mathcal I}}
\def\calJ{{\mathcal J}}

\def\calM{{\mathcal M}}

\def\calP{{\mathcal P}}

\def\calU{{\mathcal U}}

\def\scrA{{\mathscr A}}

\def\scrJ{{\mathscr J}}

\def\scrO{{\mathscr O}}

\def\scrS{{\mathscr S}}

\def\scrU{{\mathscr U}}

\def\Gtil{\widetilde G}

\def\N{\mathbb N}

\def\Q{\mathbb Q}
\def\R{\mathbb R}

\def\T{\mathbb T}

\def\Z{\mathbb Z}

\def\frB{{\mathfrak B}}

\def\frJ{{\mathfrak J}}

\def\frm{{\mathfrak m}}\def\frM{{\mathfrak M}}
\def\frn{{\mathfrak n}}\def\frN{{\mathfrak N}}

\def\frS{{\mathfrak S}}

\def\alp{{\alpha}}  
\def\bet{{\beta}}
\def\gam{{\gamma}} \def\Gam{{\Gamma}} 
\def\del{{\delta}} \def\Del{{\Delta}}
\def\eps{\varepsilon}
\def\zet{{\zeta}}  

\def\tet{{\vartheta}}  

\def\kap{{\kappa}}
\def\lam{{\lambda}}

\def\sig{{\sigma}} \def\Sig{{\Sigma}}
\def\vrho{\varrho}
 
\def\vphi{{\varphi}}
\def\ome{{\omega}} 

\def\balp{{\boldsymbol \alpha}}
\def\bbet{{\boldsymbol \beta}}
\def\bgam{\boldsymbol \gamma}

\def\bzet{{\boldsymbol \zeta}}

\def\btet{{\boldsymbol \vartheta}}

\def\bxi{{\boldsymbol \xi}}

\def\bphi{{\boldsymbol \varphi}}

\def\d{{\partial}}
\def\bzer{\boldsymbol 0}

\def\implies{\Rightarrow}

\def\le{\leqslant} \def\ge{\geqslant}

\def\d{{\,{\rm d}}}
\def\mmod#1{\;(\mathrm{mod}\;{#1})}

\DeclareMathOperator{\vol}{vol}

\DeclareMathOperator{\li}{li}
\DeclareMathOperator{\supp}{supp}

\title[A minimalist approach to the circle method]{A minimalist approach to the circle method\\and Diophantine problems over thin sets}
\author[K. D. Biggs]{Kirsti D. Biggs}
\address{Department of Mathematics, Uppsala University, 75106 Uppsala, Sweden.\newline
	 {\tt kirsti.biggs@math.uu.se}}
\author[J. Brandes]{Julia Brandes}
\address{Mathematical Sciences, University of Gothenburg and Chalmers University of Technology, 41296 Gothenburg, Sweden.
	 {\tt brjulia@chalmers.se}}


\begin{document}
\maketitle

\begin{abstract}
	This paper studies the minimal conditions under which we can establish asymptotic formul{\ae} for Waring's problem and other additive problems that may be tackled by the circle method. We confirm in quantitative terms the well-known heuristic that a mean value estimate and an estimate of Weyl type, together with suitable distribution properties of the underlying set over a set of admissible residue classes, are sufficient to implement the circle method. This allows us to give a rather general proof of Waring's problem which is applicable to a range of sufficiently well-behaved thin sets, such as the ellipsephic sets recently investigated by the first author. 
\end{abstract}

	\section{Background and teaser}

	Let $n$ be a natural number. In how many ways can $n$ be written as a sum of elements in a given subset $\calA \subseteq \N$? This simple question lies at the crossroads of various fields in pure mathematics: When the subset is the set of primes, this question amounts to Goldbach's problem. In the case of $k$th powers, it is an instance of Waring's problem, and for more general sets it is a question under active investigation in additive combinatorics. Moreover, if the elements of the subset are equipped with complex weights, the same problem becomes a question in discrete harmonic analysis. 
	
	In the manuscript at hand, we study Waring's problem and related questions in situations in which the variables are restricted to a fairly general subset of the integers. In particular, we explore what properties are sufficient if one wants to derive asymptotic formul{\ae} for the number of representations of any sufficiently large integer as a sum of $k$th powers of elements in a given subset of the integers. This can be viewed as a generalisation of work by Br\"udern \cite{B} on binary additive forms, in which he gives sufficient conditions on the sets $\calA$ and $\calB$ such that the equation $a+b=n$ has a solution with $a \in \calA$ and $b \in \calB$ for all sufficiently large integers $n$. 
	Our methods can also be used to derive asymptotic formul{\ae} for mean values of exponential sums in the supercritical range, with potential applications to questions in discrete analysis.  
	
	Historically, Waring's problem has mostly been studied over the integers. Write $G(k)$ for the least integer $s$ having the property that any sufficiently large $n \in \N$ has a representation as a sum of $s$ positive $k$th powers. Similarly, write $\Gtil(k)$ for the least integer $s$ such that the number of representations of $n$ as a sum of $s$ positive $k$th powers satisfies the anticipated asymptotic formula. There is a huge body of research aiming at reducing the upper bounds of $G(k)$ and $\Gtil(k)$, with some of the strongest available bounds due to Wooley \cite{W:95}, \cite{W:12} (for a more complete bibliography see also the survey paper by Vaughan and Wooley \cite{VW:survey}). 
	
	Waring's problem is generally handleable over sets with positive lower density (see Theorem~1.3 in \cite{BGV}, and in particular Theorem~1.3 in \cite{salmensuu}), and instances of this are given in the case of Beatty sequences \cite{BGV} or numbers with mild digital restrictions \cite{TT}. Moreover, Salmensuu \cite{salmensuu} used the transference principle from additive combinatorics to show that $k^2+k+1$ summands are sufficient to represent every large integer as soon as the lower density of the set $\calA$ is at least roughly $2^{-1/k}$ for large enough $k$. While the number of summands is larger than the $\sim k \log k$ that are required over the integers, it is essentially the same that is needed to get an asymptotic formula for the number of such representations. 
	
	The study of Waring's problem over thin sets is considerably harder, and has for the longest time focussed mainly on the case of primes, in which case it is known as the Waring--Goldbach theorem. Some bibliography of results in the field can be found in a survey paper by Kumchev and Tolev \cite{KT:survey}.  Beyond these, only very few thin sets have been investigated, including both integers \cite{AG16:Waring-PShap} and primes \cite{AG18a:Waring-Goldbach-PShap,AG18b:Waring-Goldbach-PShap} of Piatetski-Shapiro type, or sets whose elements have a representation as a sum of powers in their own right \cite{pliego-3cub, pliego-gen}. An interesting case is that of smooth numbers $\scrA(Q,P) = \{x \le P: p|x \implies p<Q\}$. When $Q$ is a small power of $P$, this set is of positive lower density, and the treatment of Waring's problem in this case is classical (see \cite{V89:iterative}) and plays a crucial role in obtaining the best possible bounds for $G(k)$. If, however, $Q=(\log P)^\kap$ for some $\kap>1$, then the cardinality of $\scrA(Q,P)$ grows like $P^{1-1/\kap+\eps}$, and it has been shown by Drappeau and Shao \cite{DS} that the analogue of Waring's problem continues to hold over these thin sets whenever $\kap$ is sufficiently large. In a similar vein, Wooley investigated a weighted version of Waring's problem that is skewed towards small solutions \cite{W:lightweight} and thus to some measure mimics the behaviour of a thin set. This last paper is related to work of Vu \cite{Vu}, later refined by Wooley \cite{W:vu}, in which they applied a probabilistic approach to show that whenever Waring's problem over the integers is solvable for a given number $n$, then there exists a suitably thin subset of the integers over which Waring's problem also has a solution, requiring only two additional variables. 
	
 	It is worth pointing out that in all of the instances mentioned the strategy to analyse the major arcs is typically to adapt the standard techniques developed for the integer case, and is thus typically heavily informed by the precise shape of the set under investigation. Our contribution lies in the fact that we propose a unified treatment of the major arc analysis, subject only to the same requirements that are already necessary in order to control the minor arcs, and applicable for all sets that satisfy some suitable distribution properties over residue classes. 
	
	Since stating our most general theorems requires a certain amount of notation, we defer the precise statements of our results to the next section and focus instead on presenting some of the most interesting consequences in order to give a flavour of the conclusions we are able to obtain. Let $\calA \subseteq \N$ be an infinite set\footnote{The notation introduced here will be generalised in the next section.}, and put 
	$$
		A(X) = \#\{x \in \calA, x \le X\}.
	$$
	For some of our applications it is useful to define the logarithmic lower density 
	\begin{align}\label{0.def-lam}
		\lam = \lam(\calA)= \liminf_{X \to \infty} \frac{\log A(X)}{\log X},
	\end{align}
	so that $A(X) \gg X^{\lam-\eps}$ for all sufficiently large $X$. 
	
	We write $R_{s;k}(n;\calA)$ for the number of representations of a natural number $n$ in the shape 
	\begin{align}\label{0.Waring-eq}
		n=x_1^k + \ldots + x_s^k,
	\end{align}
	where $x_1, \ldots, x_s \in \calA$. We also write $R_{s,1;k}(n;\calA)$ for the number of representations 
	\begin{align}\label{0.Waring-eq-2}
		n=x_1^k+\ldots+x_s^k+y^k,
	\end{align}
	where $x_1, \ldots, x_s \in \calA$ as before, but we allow one variable $y$ to range over all positive integers. For these counting functions we expect asymptotic formul{\ae} of the shape 
	\begin{align*}
		R_{s;k}(n;\calA)=A(n^{1/k})^s n^{-1} (\calC_{\calA}(n)+o(1))
	\end{align*}
	and
	\begin{align*}
		R_{s,1;k}(n;\calA)=A(n^{1/k})^s n^{1/k-1} (\calC^*_{\calA}(n)+o(1)),
	\end{align*}
	respectively, where $\calC_{\calA}$ and $\calC^*_{\calA}$ are non-negative constants having an interpretation in terms of the solution densities of the local analogues of~\eqref{0.Waring-eq} and~\eqref{0.Waring-eq-2}. Also, it will be useful to define the parameter $\vrho_0=\vrho_0(k)$ via 
	\begin{align}\label{0.rho}
		\vrho_0(k)^{-1} = \min \{2^{k-1}, (k-1)(k-2) + 2 \lfloor \sqrt{2 k}\rfloor\}.
	\end{align}
	In this notation, we have a version of Waring's problem for fairly general sets, provided that they satisfy certain distribution properties and are not too thin. 
	
	\begin{theorem}\label{T0:Akshat}
		Let $\calA = \{x_1, x_2, \ldots\} \subseteq \N$ be a set having the property that for every $q \in \N$ and every $b \in \Z/q\Z$ there exists a number $\kap(q,b) \ge 0$ with the property that
		\begin{align}\label{T0:Akshat-cond2}
			\# \{x \in \calA \cap [1,X]:\, x \equiv b \mmod q\} = \kap(q,b) A(X) + O(1)
		\end{align}
		Suppose further that we are in one of the following two cases:
		\begin{enumerate}[(A)]
			\item\label{T0:Akshat-VMVT} We have $\lam(\calA)^{-1} < 1+\vrho_0(k) \vrho_0(k+1)/5$ and $s \ge \vrho_0(k+1)^{-1} + 1$.
			\item\label{T0:Akshat-convex} The set $\calA$ satifies $\lam(\calA)^{-1} < 1+\vrho_0(k)/(10k)$ and  
			\begin{align}\label{T0:Akshat-cond1}
				x_{n+1}-x_n \le x_{n+2}-x_{n+1} \qquad \text{for all $n \in \N$,}
			\end{align} 
			and we have $s \ge C 2^k \log k$ for some computable constant $C$.
		\end{enumerate}
		Then we have 
		\begin{align*}
		R_{s,1;k}(n;\calA)=A(n^{1/k})^s n^{1/k-1} (\calC^*_{\calA}(n)+o(1)).
		\end{align*}		
	\end{theorem}
	
	To the authors' knowledge, this is the first time a version of Waring's problem over thin sets has been stated for sets of a similar generality. Note in particular that we do not assume $A(X)$ to satisfy any asymptotic formula. Instead we have only the density condition, which for $k\ge6$ is of the shape $\lambda(\calA) \ge 1-ck^{-4}$ in case \eqref{T0:Akshat-VMVT} and of size $\lambda(\calA) \ge 1-ck^{-3}$ in statement \eqref{T0:Akshat-convex}, where $c<1$ is some absolute constant. These density conditions can be relaxed at the expense of including more variables running over the whole of $\N$. We also remark that the error terms in condition~\eqref{T0:Akshat-cond2} can be relaxed, depending on the properties of the set $\calA$. Similarly, in part \eqref{T0:Akshat-convex} of the statement there is a trade-off between the specific shape of the respective bounds on the logarithmic density $\lambda(\calA)$ and the number of variables $s$, each of which can be sharpened at the expense of the other. 
	
	The main ingredient in the conclusion of Theorem~\ref{T0:Akshat} is Vinogradov's mean value theorem \cite[Theorem~1.1]{BDG} (see also \cite[Theorem~1.1]{W:NEC}) for part \eqref{T0:Akshat-VMVT}, and a recent result by Mudgal \cite{mudgal-convex} on higher energies of $k$-convex sets in part \eqref{T0:Akshat-convex}. In the first case, when $\lambda(\calA)$ is not too far from $1$, one does not lose too much by summing over all integers rather than only those in $\calA$; this is the origin of the density condition. Meanwhile, Mudgal's result provides a Hua-type bound for convex sets, which requires more variables but applies to potentially more general sets. We use his result in a somewhat wasteful way, yet it permits sets of a smaller logarithmic density than that which follows from Vinogradov's mean value theorem. On the other hand, in such settings where the underlying set is better understood, one would expect to get accordingly stronger results. 
    
    One set where we have such estimates is that of ellipsephic numbers. We refer to a set as a $B_m$ Sidon set if it has the property that 
	\begin{equation}\label{0:additivity}
		\#\{(a_1,\dots,a_m)\in \calD_p^m\mid a_1+\ldots+a_m=n\} \le m!
	\end{equation}
	for all $n \le mp$. In other words, a set $\calD$ is a $B_m$ Sidon set if the representation of any $n$ as a sum of $m$ elements in $\calD$ is unique up to permutations. 
	For a fixed $B_m$ Sidon set $\calD_p \subseteq \Z/p\Z$ we then define the set $\calE_p$ of $p$-ellipsephic numbers as the set 
	$$
		\calE_p = \{x \in \N: x = a_0 + a_1 p + a_2 p^2 + \ldots, \qquad a_i \in \calD_p\}.
	$$
	Observe that 
	$$
		\# \calE_p(p^h) = (\# \calD_p)^h = (p^h)^{\log \# \calD_p / \log p},
	$$
	and it is not hard from this to show that $\lam(\calE_p) =\log \# \calD_p / \log p$. Note that the property \eqref{0:additivity} implies that
	$$
		A(X)^m \asymp \#(m \calE_p(X)) \ll X,
	$$ 
	where the expression $m \calE_p(X)$ is to be interpreted as the sumset consisting of all elements $a_1 + \ldots + a_m$ with $a_1, \ldots, a_m \in \calE_p(X)$. 
	Consequently, the digit set of an ellipsephic set of logarithmic lower density $\lam$ can be a $B_m$ Sidon set only when $m \le 1/\lam$. Under such conditions, our methods lead to the following result. 
	\begin{theorem}\label{T0:Kirsti}
	Let $p>k$ be a prime, and let $\calD_p$ be a $B_m$ Sidon set with $0 \in \calD_p$.
    Let $\calE_p$ be the ellipsephic set whose digits in base $p$ lie in $\calD_p$. Suppose further that 
    \begin{align}\label{0.ellips-cond}
        m > \frac{1}{\lambda}-\frac{2\vrho_0}{5k(k+1)},
    \end{align}
    for $\lambda$ and $\vrho_0$ defined as in \eqref{0.def-lam} and \eqref{0.rho} respectively. Whenever $s\geq 2t_0=m k(k+1)$, we have the asymptotic formula
    \begin{align*}
        R_{s,1;k}(n; \calE) = A(n^{1/k})^sn^{1/k-1} (\calC_{\calE}^*(n) + o(1)).
    \end{align*}
    Moreover, the constant $\calC_{\calE}^*(n)$ has an interpretation in terms of solution densities of the underlying equation over all completions of $\Q$. 
	\end{theorem}

In this result, the main ingredient is the result by the first author \cite{biggs1,biggs2} on Vinogradov's mean value theorem for ellipsephic numbers. An alternative approach, based on $\ell^2$-decoupling rather than efficient congruencing, is due to Chang et al.~\cite{Chang++}. As stated, our result holds only for such ellipsephic sets whose generating set $\calD_p$ is a $B_m$ Sidon set for $m$ very close to $1/\lam(\calE_p)$. Our methods yield results for smaller values of $m$ as well, but as they stand at the moment that would require more than just one supplementary variable that runs over the full set of integers. In a similar way, it is apparent from \cite{biggs1} and \cite{Chang++} in the case $k=2$ that the condition of being a $B_m$ Sidon set can be somewhat relaxed at the expense of incurring some losses in the bound, and those comparatively weaker bounds will again require additional supplementary variables.

The alert reader will have noticed that both Theorem~\ref{T0:Akshat} and~\ref{T0:Kirsti} are concerned with the more relaxed version of Waring's problem~\eqref{0.Waring-eq-2} in which one of the variables is allowed to run over the full set of integers. This is a consequence of the fact that estimates of Weyl type are usually not available for completely general sets. In such situations where bounds of this type are known, however, we do indeed obtain bounds on the counting function $R_{s;k}(n;\calA)$ associated with the original Waring problem over thin sets. 

Similarly, in both of the results discussed so far, the underlying sets have excellent distribution properties over arithmetic progressions. This is not the case for all numbers, and a notorious example of a set that is comparatively poorly understood along arithmetic progressions is that of the primes. Our last teaser result therefore concerns the set of primes. Whilst our result here is essentially identical to what is obtained by feeding recent progress on Vinogradov mean values into the same strategy employed by Hua \cite{Hua} (see also recent work by Talmage \cite{Talmage}), it can nonetheless serve as an illustration that our theorems accommodate even the notoriously weak error terms that arise in the prime setting. 
\begin{theorem}\label{T0:Goldbach}
	Let $\calP$ denote the set of primes, and assume that 
	$$
		s \ge k(k-1) + 2\lfloor \sqrt{2k+1} \rfloor + 1.
	$$ 
	Under such circumstances we have 
	\begin{align*}
		R_{s;k}(n;\calP) = \frac{\Gam(1/k)^s}{\Gam(s/k)} \frS_{\calP}(n) \frac{n^{s/k-1}}{(\log n)^s} (1 + o(1)),
	\end{align*}
	and the singular series $\frS_{\calP}(n)$ is absolutely convergent.	
\end{theorem}
The singular series is the same as that known in the Waring--Goldbach theorem and is given explicitly for instance in \cite[Satz 11]{Hua}.
	
Our methods also apply to systems of equations, including questions on asymptotic formul{\ae} for mean values of Vinogradov shape. In fact, although we do not go into details in this work, a very straightforward application of our methods yields an improved version of Theorem~1.1 and its corollaries of the second author's previous work with Parsell \cite{BP1}, as well as a much simplified treatment of the singular series in work of the second author and Wooley \cite{BW2}. \smallskip
	
The main technical inputs in our results are two. On the one hand, we take inspiration from an idea going back to Heath-Brown and Skorobogatov \cite{HB-Skoro} in combination with methods commonly applied in the study of Diophantine equations of general shape to establish convergence of both the singular integral and the singular series using only a mean value estimate and a bound of Weyl type. On the other hand, we show that for any (weighted) set $\calA$ there exists a suitable continuous and almost everywhere differentiable function $\Psi$ approximating the counting function $A(X)$ in terms of which the function $R_{s;k}(n;\calA)$ can be measured. When some established approximating function $\Psi(X)$ of $A(X)$ is already available, it can be used, but for more general sets $\calA$ the function $\Psi$ we choose is in fact surprisingly simple. It turns out that the methods are fairly flexible in terms of the function $\Psi$, whose explicit shape will ultimately only inform the precise expression for the singular integral. Naturally, the interpretation of the singular integral as a density of solutions over the real unit cube needs to be adapted to reflect the distribution of $\Psi$ over the available range of integers. Nonetheless, we are able to show that such an interpretation is justified even in the very general setting under consideration. 
	
We are aware that neither of these two technical features are entirely novel in themselves, although the idea of \cite{HB-Skoro} has to our knowledge never been applied to the singular series, and neither are we aware of a treatment that involves near-arbitrary continuous approximations of the counting functions. Rather, the purpose of this memoir is to show that these two features are sufficient to generalise the circle method to ground sets of considerable generality, as long as some distribution properties over residue classes are satisfied, and ultimately allows us to establish results like the very general one of Theorem~\ref{T0:Akshat}.
\medskip

\textbf{Acknowledgments.} We are very grateful to J\"org Br\"udern, Olivier Robert and Igor Shparlinski for valuable discussions and suggestions that improved the paper. 
	
During the production of this manuscript, the first author was supported by postdoctoral grant no.~2018.0362 of the Knut and Alice Wallenberg Foundation, in the name of the second author. In addition, the second author held a Starting Grant of the Swedish Science Foundation (Vetenskapsr{\aa}det) under grant agreement no.~2017-05110, as well as, in the final stages, Project Grant no.~2022-03717 of the same sponsor. Part of the work on the manuscript was done while the second author was in residence at the Max Planck Institute of Mathematics in Bonn, whose generous support and excellent working conditions are also gratefully acknowledged.
	
\vfill

\section{Statement of general results}

\textbf{Notation.}  We use vector notation liberally in a manner that is easily discerned from the context. In particular, when $\bb$ denotes the integer tuple $(b_1, \dots, b_n)$, we write $(q, \bb) = \gcd(q, b_1, \dots, b_n)$.
Throughout, the letter $\eps$ will be used to denote an arbitrary, but sufficiently small,  positive number, and we adopt the convention that whenever it appears in a statement, we assert that the statement holds for all sufficiently small $\eps>0$. Usually, we take $X$ to be a large positive number which, just like the implicit constants in the notations of Landau and Vinogradov, is permitted to depend at most on $s$, $k$ (or the system $\bphi$ to be introduced in Section~\ref{subs: MVs} below), and $\eps$. We employ the non-standard notation $f(X) \llcurly g(X)$ to mean $f(X) = o(g(X))$.
	
We also take the opportunity to refine some of the notation introduced in the introduction. One feature is that we allow the set $\calA$ to be weighted. Let $(a_n)_{n \in \N}$ denote a sequence of non-negative weights summing to infinity. We then refer to the sequence $(a_n)$ and the weighted set $\calA$ interchangeably. Put $\calA(X) = (a_n)_{1 \le n \le X}$ and $A(X) = \sum_{n \le X}a_n$. We abbreviate $A(X)=A$ whenever there is no danger of confusion. Occasionally, we may use the $\ell^p$-norm of $(a_n)$ on the interval $[1,X]$, which then is denoted by $\|a\|_{\ell^p(X)}$. Note that in this notation we have $A(X) = \|a\|_{\ell^1(X)}$.
	
It will be useful on occasion to retain notation for the unweighted analogue of the weighted set $\calA$. When this necessity arises, we denote the set with a superscript circle, so that   
$$
	\calA^\circ = \supp(a_n) = \{n \in \N: a_n >0\}.
$$

\subsection{Results on Waring's problem}
For $n \in \N$ define the function
\begin{align*}
	R_{s;k}(n; \calA) = \sum_{\bx \in \N^s} a_{x_1} \cdots a_{x_s} \mathbbm{1}_{[x_1^k + \ldots + x_s^k=n]},
\end{align*}
where $\mathbbm{1}_{[x=y]}$ takes the value $1$ when $x=y$, and $0$ otherwise. 
Thus, $R_{s;k}(n; \calA)$ counts the number of solutions $\bx \in \calA$ to the equation 
\begin{align}\label{1.waringeq}
	x_1^k + \ldots + x_s^k =n,
\end{align}
where each variable $x_i$ is counted with the weight $a_{x_i}$. In particular, when $(a_n)$ is the indicator function on a subset $\calA^\circ \subseteq \N$, then $R_{s;k}(n; \calA)$ describes the number of ways in which $n$ may be written as a sum of $s$ positive $k$-th powers of elements in $\calA^\circ$. 

As is usual in problems of this flavour, we approach the question of determining $R_{s;k}(n;\calA)$ by the circle method. Set $\T = \R /\Z$, and for $\alp \in \T$ define the exponential sum 
$$
	f_k(\alp)=f_{k, \calA}(\alp;X) = \sum_{x \le X}a_x e(\alp x^k),
$$
as well as the mean value 
$$
	I_{t,k}(X; \calA) = \int_\T |f_{k, \calA}(\alp; X)|^{2t} \d \alp.
$$
\begin{definition}\label{def-V1}
	We say that the exponential sum $f_{k,\calA}(\alp;X)$ satisfies Hypothesis~\eqref{V1} with parameters $t_0$ and $\Del = \Del(X)$ if 
	\begin{align}\label{V1}\tag{V'}
		I_{t,k}(X;\calA) \ll A(X)^{2t}X^{-k}\Del(X) \qquad \text{for all $t \ge t_0$}.
	\end{align}
\end{definition}
Note that Property~\eqref{V1} is always true when $\Del = X^k$. On the other hand, when $\|a\|_{\ell^2(X)}^2 \asymp X$, we know as a consequence of recent progress on Vinogradov's mean value theorem \cite{BDG,W:NEC} that Property~\eqref{V1} is satisfied for $\Del = X^\eps$ and $t_0 = k(k+1)/2$. 

We also need pointwise information on the sum $f_{k,\calA}(\alp;X)$. For this purpose, define the major arcs 
\begin{align*}
	\frM_X(Q) = \bigcup_{q \le Q} \{ \alp \in \T: \|\alp q\| \le QX^{-k}\},
\end{align*}
and the corresponding minor arcs $\frm_X(Q) = \T \setminus \frM_X(Q)$. Just like in our other notational conventions, we will drop the subscript $X$ in most situations, keeping it only when it seems necessary to avoid misunderstandings. 
\begin{definition}\label{def-W1}
	We say that the exponential sum $f_{k,\calA}(\alp;X)$ satisfies Hypothesis~\eqref{W1} with parameters $\vrho \ge 0$, $Q_W = Q_W(X)$ and $L = L(X) \ll X^\eps$ when for all $Q \le Q_W$ we have 
	\begin{align}\label{W1}\tag{W'}
	\sup_{\alp \in \frm_X(Q)} |f_{k, \calA}(\alp;X)| \ll A(X) Q^{-\vrho} L(X).
	\end{align}
\end{definition}
Practitioners of the circle method will easily recognise bounds of the shape~\eqref{V1} and~\eqref{W1} as the main ingredients for a minor arc analysis in Waring's problem. We note here that Hypothesis~\eqref{W1} is trivial unless the function $Q_W$ is an increasing function tending to infinity. 

In order to understand the major arcs, we also need information on how the set $\calA$ is distributed over residue classes modulo $q$, for a suitable range of $q$. 
For $b \in \{0,1,\ldots,q-1\}$ put 
$$
	A(q,b;X)= \sum_{\substack{n \le X \\ n \equiv b \mmod q}}a_n.
$$
We then need the following distribution property.
\begin{definition}\label{def-E}
	We refer to a set $\calA$ as distributed to the level $Q_D = Q_D(X)$ when for each $q \le Q_D$ and $b \mmod q$ there is a number $\kap(q,b) \ge 0$ such that the asymptotic formula 
	\begin{align}\label{D}\tag{D}
		A(q,b;X)= 
		\kap(q,b) A(X)+ O(E(X)) 
	\end{align}
	is satisfied with an error term $E(X) = o(A(X))$. 
\end{definition}
Again, we may assume without loss of generality that $Q_D$ is a continuously increasing unbounded function. 
Our condition~\eqref{D} is reminiscent of its namesake condition in \cite{B}, where it was the main requirement on the sets under investigation. Our formulation is somewhat more permissive than Br\"udern's in that we allow weights and the set $\calA$ is not required to be of positive density.

For future reference we note a simple additivity property of the functions $\kappa(q,b)$. Obviously, for $q,q' \in \N$ and $b \in \Z/q\Z$ one has 
\begin{align*}
	A(q,b;X)=\sum_{\substack{c \mmod{qq'} \\ c \equiv b \mmod q}}A(qq',c;X).
\end{align*}
Using Condition~\eqref{D} on either side, dividing by $A(X)$ and letting $X$ tend to infinity, we deduce that 
\begin{align}\label{2:kappa-add}
	\kappa(q,b) =\sum_{\substack{c \mmod{qq'} \\ c \equiv b \mmod q}} \kappa(qq',c).
\end{align} 
In particular, we necessarily have 
\begin{align}\label{2:kappa-add-1}
	\sum_{b \mmod q}\kap(q,b)=1.
\end{align}

We are now prepared to state our first main result. Define 
\begin{align}\label{1.Y}
	Y(X)=\min\left\{Q_D(X), Q_W(X), \left(\frac{A(X)}{E(X)}\right)^{1/5}, X^{1/5} \right\},
\end{align}
then we have the following Waring-type result. 
\begin{theorem}\label{T:Waring}
	Suppose that $\calA$ is such that the Hypotheses~\eqref{D}, \eqref{V1} and~\eqref{W1} are satisfied with parameters $t_0$, $\Del$, $\vrho$, $Q_D$, $Q_W$ and $L$, and assume that $L(X) \ll Y(X)^{\vrho-\del}$ for some $\del>0$.  Take $\sig_0$ minimal such that for all $\eps>0$ one has 
	\begin{align}\label{1.sig}
		\Del(Z)  \ll \left(\frac{Y(Z)^{\vrho}}{ L(Z)}\right)^{\sig_0 + \eps}
	\end{align}
	uniformly in $Z$. Then for any integer $s > 2t_0+\sig_0$ and any $n \in \N$ we have an asymptotic formula of the shape 
	\begin{align}\label{1.warasymp}
		R_{s;k}(n; \calA) = A(n^{1/k})^sn^{-1} (\calC_{\calA}(n) + o(1))
	\end{align}
	for a non-negative constant $\calC_{\calA}(n)$. In particular, we have $\Gtil_{\calA}(k) \le \lfloor2t_0 + \sig_0\rfloor + 1$. 
\end{theorem}

In many cases the constant $\calC_{\calA}(n)$ may be analysed further. The relevant condition is the following. 
\begin{definition}\label{def-C}
	Suppose that the set $\calA$ is distributed to the level $Q_D = Q_D(X)$. We say that $\calA$ satisfies Hypothesis (C) if the coefficients $\kap(q,b)$ satisfy 
	\begin{align}\label{C}\tag{C}
		\kap(qq',qb'+q'b)=\kap(q,q'b)\kap(q',qb') \qquad \text{for all $ (q,q')=1$}.
	\end{align}
\end{definition}

We will see in the proof that if \eqref{C} holds, the constant $\calC$ can be written as a product over local factors in the shape 
$$
	\calC_{\calA}(n) = \chi_{\infty}(n;\calA) \prod_{p \text{ prime}} \chi_p(n;\calA).
$$
For the interpretation of these factors, we observe that the distribution property~\eqref{D}, and in particular \eqref{2:kappa-add}, implies that the set $\calA$ can be embedded in the $p$-adic numbers $\Z_p$ for each $p$. In the special case when for all $q$ the set $\calA^\circ$ is equidistributed over a specified set $\calA_q \subseteq \Z/q\Z$ of admissible residue classes modulo $q$, the inverse limit $\varprojlim \calA_{p^h}$ produces a $p$-adic Cantor set $\Z_p(\calA)$. Under such circumstances, each factor $\chi_p(n;\calA)$ denotes the density of solutions of~\eqref{1.waringeq} as the variables $x_i$ range over the set $\Z_p(\calA)$. In particular, if $\calA^\circ$ is equidistributed over arithmetic progressions in general, this factor is identical to the one occurring in the integral case $\calA=\N$.

In order to understand the factor at infinity, let $\Psi$ be a continuous, piecewise differentiable, increasing function satisfying $\Psi(X) - A(X) \ll E(X)$, where the error $E(X)$ is the same as in Hypothesis~\eqref{D}. For many sets $\calA$, an asymptotic formula of this shape is known, but we will show in Section~\ref{S3} below how such a function can be constructed for any set $\calA$ satisfying condition~\eqref{D}. With this function at hand, we can define the $\calA$-weighted measure $\sig_{\calA}$ on $[0,1]^s$ by setting 
\begin{align*}
	\d \sig_{\calA}(\bxi) = \prod_{i=1}^s  \frac{n^{1/k}}{\Psi(n^{1/k})}\d \Psi(n^{1/k}\xi_i),
\end{align*}
where the differential is to be understood in the Riemann--Stieltjes sense. This measure is a smoothed-out version of the point measure on $\calA(n^{1/k})$, rescaled to be a probability measure on the unit hypercube. In this notation, the factor at infinity can be shown to be of the shape 
\begin{align*}
	\chi_{\infty}(n;\calA) = \int_\R \int_{[0,1]^s} e(\alp (\xi_1^k + \ldots + \xi_s^k - 1)) \d \sig_{\calA}(\bxi) \d \alp,
\end{align*}
and consequently has an interpretation as a weighted solution density of the equation $\xi_1^k + \ldots + \xi_s^k = 1$ in the unit cube, with the measure reflecting the distribution of the set $\calA(n^{1/k})$ inside $[1,n^{1/k}]$. When the function $\Psi$ is sufficiently well-behaved, this factor can by standard arguments be expressed in terms of Gamma functions, but we do not require that here.

As adumbrated above, Theorem~\ref{T:Waring} shows that, indeed, a version of Hua's lemma and a pointwise bound of Weyl type are sufficient to control the contribution from the major arcs. In particular, we do not require any of the delicate bounds on complete exponential sums that are commonly used in arguments of this flavour, or even a good understanding of the asymptotic behaviour of $A(X)$. The fact that the only input we require for this result, apart from the standard tools for a minor arc analysis (which, however, may be highly non-trivial in themselves), is the requirement that the set $\calA$ be suitably distributed over a permissible set of residue classes substantiates the received folklore that the minor arc analysis really does present the bottleneck for most circle method applications. 
Since some of the hypotheses of the theorem and in particular condition~\eqref{1.sig} may be somewhat hard to parse, it may be useful for the reader to state the shape Theorem~\ref{T:Waring} will take in a typical situation. 

\begin{corollary}\label{cor-W}
	Let $\calA$ be a set with positive logarithmic lower density, and suppose that it satisfies the assumptions of Theorem~\ref{T:Waring} with $E(X)\ll 1$, $Q_D(X) \asymp Q_W(X) \asymp X$, as well as $\Del(X) \ll \min(X,A(X))^\ome$ for some $\ome \ge 0$. Then~\eqref{1.warasymp} is satisfied whenever $\sig_0 \ge 5 \ome /  \vrho$. In particular, when $\ome < \vrho/5$, then~\eqref{1.warasymp} holds with $s = 2t_0+1$.  
\end{corollary}
In this formulation it becomes clear that, should  Property~\eqref{V1} be known with $\Del \ll X^\eps$, then any power saving in the Weyl bound will lead to the expected asymptotic formula, provided that the set $\calA$ is equidistributed over an admissible set of residue classes modulo each $q$. This reproduces the well-known behaviour of the case $\calA = \N$. In fact, a careful reading of the proof will show that under the hypotheses of Corollary~\ref{cor-W} the error term $o(1)$ in~\eqref{1.warasymp} can be replaced by a power-saving error term $O(n^{-\del})$ for some $\del>0$.

\medskip
\subsection{Waring's problem with additional variables}
It is not uncommon to find oneself in a situation where either no bound of Weyl type is available, or else it is too weak for the purposes at hand. In such situations, one typically makes recourse to a substitute by counting solutions to~\eqref{1.waringeq} in which a few of the variables are allowed to run over the full set $\N$ rather than a subset $\calA$. We define 
\begin{align*}
	R_{s,u;k}(n; \calA) =  \sum_{\substack{\bx \in \N^s \\ \by \in \N^u}} a_{x_1} \cdots a_{x_s} \mathbbm{1}_{[x_1^k + \ldots + x_s^k + y_1^k + \ldots + y_u^k=n]},
\end{align*} 
so that $R_{s,u;k}$ counts the number of solutions to the equation 
\begin{align}\label{1.waringeq2}
	x_1^k + \ldots + x_s^k + y_1^k + \ldots + y_u^k =n,
\end{align}
where $x_1, \ldots, x_s$ run over the set $\calA$ and are counted with the concomitant weights $a_{x_j}$, whereas $y_1, \ldots, y_u$ run over the positive integers and are counted with weight $1$. In such a situation, we put
\begin{align}\label{1.Y2}
	Y_2(X)=\min\left\{Q_D(X), \left(\frac{A(X)}{E(X)}\right)^{1/5}, X^{1/5} \right\},
\end{align}
and recall the definition of $\vrho_0$ from~\eqref{0.rho}. 
Our result is as follows.

\begin{theorem}\label{T:Waring2}
	Suppose that $\calA$ is such that Properties~\eqref{D} and~\eqref{V1}  are satisfied with parameters $t_0$, $\Del$, and $Q_D$. Assume that $Y_2(X)^{\vrho_0-\del} \gg \log X$ for some $\del>0$, and take $\sig_0$ minimally such that the relation
	\begin{align}\label{1.sig2}
		\Del(Z)  \ll \left(\frac{Y_2(Z)^{\vrho_0}}{\log Z}\right)^{\sig_0+\eps}
	\end{align}
	is satisfied uniformly in $Z$ for all $\eps>0$. Then for any integers $s \ge 2t_0$ and $u > \sig_0$, and for any $n \in \N$, we have an asymptotic formula of the shape 
	\begin{align}\label{1.warasymp2}
		R_{s,u;k}(n; \calA) = A(n^{1/k})^sn^{u/k-1} (\calC_{\calA}^*(n) + o(1)),
	\end{align}
	where $\calC_{\calA}^*(n)$ is a non-negative constant. If also \eqref{C} is satisfied, this constant has an interpretation in terms of a product of solution densities of the local analogues of~\eqref{1.waringeq2}. 
\end{theorem}
The most well-known versions of Theorem~\ref{T:Waring2} are concerned with the case when $\calA$ is a suitable set of smooth numbers, in which situation the conclusion of the theorem leads to some of the strongest known bounds for the size of $G(k)$ in Waring's problem. However, other sets are also possible and of interest, such as for instance Pliego's recent work \cite{pliego-3cub,pliego-gen} involving a set $\calA$ of numbers that can be represented as sums of $\ell$th powers. 
\begin{corollary}\label{cor-W2}
	Let $\calA$ be a set with positive logarithmic lower density, and suppose that it satisfies the assumptions of Theorem~\ref{T:Waring2} with $E(X)\ll 1$, $Q_D(X)  \asymp X$, as well as $\Del(X) \ll \min(X,A(X))^\ome$ for some $\ome \ge 0$. Then~\eqref{1.warasymp2} is satisfied whenever $\sig_0 \ge 5 \ome /  \vrho_0$. In particular, when $\ome < \vrho_0/5$, then~\eqref{1.warasymp2} holds with $s \geq 2t_0$ and $u\geq 1$.  
\end{corollary}

\subsection{Mean values} \label{subs: MVs}
Our methods can also be used to derive asymptotic formul{\ae} for the counting functions associated with mean values of exponential sums. Although they are applicable equally well to exponential sums of multivariate polynomials, such as also occur in certain systems of diagonal equations with repeated degrees, we focus here on the univariate case. Suppose that $\vphi_1, \ldots \vphi_r \in \Z[X]$ are polynomials with respective degrees $k_1, \ldots, k_r$, and define the exponential sum 
$$
	f_{\bphi,\calA}(\balp;X) = \sum_{x \le X} a_x e(\alp_1 \vphi_1(x) + \ldots + \alp_r \vphi_r(x)). 
$$
Although in principle a suitable adaptation of our methods also covers systems of non-homogeneous polynomials, we obtain cleaner estimates when the $\vphi_i$ are homogeneous, so for the purpose of this paper we restrict to the homogeneous case. 
In this notation, the mean value 
$$
	I_{s,\bphi}(X; \calA) = \int_{\T^r} |f_{\bphi,\calA}(\balp;X)|^{2s} \d \balp
$$
has for any integer $s$ an interpretation as the weighted number of solutions to the system of equations 
\begin{align}\label{1.mv-eq}
	\vphi_j(x_1) + \ldots + \vphi_j(x_s) = \vphi_j(y_1) + \ldots + \vphi_j(y_s) \qquad (1 \le j \le r)
\end{align}
where $1 \le \bx ,\by \le X$, and where every solution $(\bx,\by) \in \N^{2s}$ is counted with the weight $a_{x_1} \cdots a_{x_s} a_{y_1} \cdots a_{y_s}$.
For future reference we remark here that it is apparent from the interpretation of the mean value $I_{s, \bphi}(X;\calA)$ that we may suppose without loss of generality that the polynomials $\vphi_j$ have no constant terms. 
Putting $K = k_1 + \ldots + k_r$, a well-known argument along the lines of that presented in \cite[Chapter~7.1]{V:HL} shows that for every $s>0$ one has the lower bound
\begin{align}\label{1.lower-bd}
	I_{s,\bphi}(X; \calA) \gg A(X)^{s}+ A(X)^{2s}X^{-K}.
\end{align}
On the other hand, when the system $\bphi$ has a non-vanishing Wronskian, it is a consequence of Vinogradov's mean value theorem \cite[Theorem~1.1]{BDG}, \cite[Theorem~1.1]{W:NEC} that for any sequence $(a_n)$ we have
\begin{align}\label{1.VMVT-bound}
	I_{s,\bphi}(X; \calA) \ll \begin{cases}\|a\|_{\ell^2(X)}^{2s}X^{\eps} & \text{if }s \le \tfrac12r(r+1)\\ \|a\|_{\ell^2(X)}^{2s} X^{s-K+\eps}& \text{if } s \ge \tfrac12 k_{\mathrm{max}}(k_{\mathrm{max}}+1), \end{cases}
\end{align}
where we have written $k_{\mathrm{max}} = \max_j k_j$. 
An application of the Cauchy--Schwarz inequality shows that $A(X) \ll \| \mathbbm{1}_{\calA}\|_{\ell^2(X)} \|a\|_{\ell^2(X)}$, where $\mathbbm{1}_{\calA}$ denotes the  indicator function on the un-weighted set $\calA^\circ$. Thus, when $\calA$ is sparse inside $\N$ so that $\| \mathbbm{1}_{\calA}\|_{\ell^2(X)}$ is much smaller than $X^{1/2}$,  then for large values of $s$ the bounds in \eqref{1.lower-bd} and \eqref{1.VMVT-bound} will not match. In the case when $\calA$ is a suitable ellipsephic set, the first author \cite{biggs1,biggs2} as well as subsequent work by Chang et al.~\cite{Chang++} have strengthened the bound in \eqref{1.VMVT-bound} to resemble more closely the expected size as indicated in \eqref{1.lower-bd}. 

Just as in the case of Waring's problem, there is a general expectation that when an upper bound of the shape 
$$
	I_{s, \bphi}(X;\calA) \ll A(X)^{2s}X^{-K+\eps}
$$
has been found for all $s$ at least as large as some parameter $t_0$, then one should be able to establish an asymptotic formula for $I_{s,\bphi}(X;\calA)$ whenever $s$ is strictly larger than $t_0$. Moreover, even a weaker bound will still lead to an asymptotic formula, albeit at the expense of requiring a correspondingly larger number of variables. We make that expectation explicit and give a general proof of this fact under fairly general assumptions. 

As in the setting with Waring's problem, the statement of our result requires some notation. In particular, we need a more general version of Hypotheses \eqref{V1} and \eqref{W1} reflecting the existence of multiple polynomials $\vphi_1, \ldots, \vphi_r$.
\begin{definition}\label{def-V}
	We say that the exponential sum $f_{\bphi,\calA}(\balp;X)$ satisfies Hypothesis (V) with parameters $t_0$ and $\Del = \Del(X)$ if 
	\begin{align}\label{V}\tag{V}
		I_{t,\bphi}(X;\calA) \ll A(X)^{2t}X^{-K}\Del(X) \qquad \text{for all $t \ge t_0$}.
	\end{align}
\end{definition}
Clearly, Hypothesis \eqref{V1} is a special case of Hypothesis \eqref{V}, applying to the situation where $r=1$ and $\vphi(x)=x^k$. 

In order to state the generalised version of condition~\eqref{W1}, define the major arcs 
$$
	\frM_X(Q) = \bigcup_{q \le Q} \{ \balp \in \T^r: \|\alp_j q\| \le QX^{-k_j} \quad (1 \le j \le r)\},
$$
and the corresponding minor arcs $\frm_X(Q) = \T^r \setminus \frM_X(Q)$. As usual, we may drop indices whose value can be inferred from context. 
\begin{definition}\label{def-W}
	We say that the exponential sum $f_{\bphi,\calA}(\balp;X)$ satisfies Hypothesis \eqref{W} with parameters $\vrho \ge 0$, $Q_W = Q_W(X)$ and $L = L(X) \ll X^\eps$ when for all $Q \le Q_W$ we have 
	\begin{align}\label{W}\tag{W}
		\sup_{\balp \in \frm_X(Q)} |f_{\bphi, \calA}(\balp;X)| \ll A(X) Q^{-\vrho} L(X).
	\end{align}
\end{definition}
Again, condition~\eqref{W1} is the special case of Hypothesis \eqref{W} that is obtained by specialising $r=1$ and $\vphi(x)=x^k$. 

Given the various parameters of Properties \eqref{D}, \eqref{V} and \eqref{W}, put
\begin{align}\label{1.Y-mv}
	Y(X)=\min\left\{Q_D(X), Q_W(X), \left(\frac{A(X)}{E(X)}\right)^{1/(2r+3)}, X^{1/(2r+3)} \right\}.
\end{align}
Again, we see that the definition of $Y(X)$ in \eqref{1.Y} just corresponds to the situation where $r=1$ here. 
In this notation, our main result for mean values is as follows. 
\begin{theorem}\label{T:VMVT}
	Suppose that $\vphi_j(x)=c_j x^{k_j}$ for $1 \le j \le r$, with $k_1 < \ldots < k_r$. Suppose further that $\calA$ is such that the conditions \eqref{D}, \eqref{V} and \eqref{W} are satisfied for $f_{\bphi,\calA}(\balp;X)$ with parameters $t_0$, $\Del$, $\vrho$, $Q_D$, $Q_W$ and $L$, and assume that $L(X) \ll Y(X)^{\vrho-\del}$ for some $\del>0$.  Define $\sig_0$ as in \eqref{1.sig}. Then for any integer $s$ with $2s > 2t_0+\sig_0$ we have an asymptotic formula of the shape 
	\begin{align}\label{1.mv-asymp}
		I_{s,\bphi}(X; \calA) = A(X)^{2s}X^{-K} (\calC_\calA + o(1)),
	\end{align}
	where $\calC_\calA$ is a non-negative constant. If, moreover, \eqref{C} holds, this constant has an interpretation in terms of a product of solution densities of suitable analogues of \eqref{1.mv-eq} over all completions $\Q_v$ of $\Q$.
\end{theorem}
Just like in the case of Theorem~\ref{T:Waring}, we will see below that when \eqref{C} holds the constant $\calC_\calA$ has a representation 
$$
	\calC_{\calA} = \chi_{\infty}(\calA) \prod_{p \text{ prime}} \chi_{p}(\calA),
$$
where $\chi_{p}(\calA)$ corresponds to the solution density of \eqref{1.mv-eq} over the set $\Q_p(\calA)$, while $\chi_{\infty}(\calA)$ measures the density of solutions of \eqref{1.mv-eq} over the real unit cube, but with respect to the weighted measure $\sig_{\calA}$ rather than the naive one.  We also remark that in the case when the $\bphi$ have lower order terms, we still obtain the expected main term and an error term of smaller magnitude, but the main term will not have an interpretation as simple as the one given in \eqref{1.mv-asymp}. 

As in the previous setting, Theorem~\ref{T:VMVT} confirms the principle that in order to establish asymptotic formul{\ae} for mean values of moments of multidimensional exponential sums,  the analysis of the minor arcs should form the bottleneck of the argument. Indeed, we see that apart from the conditions \eqref{V} and \eqref{W} controlling the minor arc bound, the only additional input that is needed is that the set $\calA$ should satisfy suitable distribution properties over residue classes. 

Since the statement of Theorem~\ref{T:VMVT} may be hard to parse, we present a simplified statement that applies to many situations of interest. 

\begin{corollary}\label{cor-MV}
	Let $\calA$ be a set with positive logarithmic lower density. Suppose that $\calA$ and $\bphi$ satisfy the assumptions of Theorem~\ref{T:VMVT} with $E(X)\ll 1$, $Q_D(X) \asymp Q_W(X) \asymp X$, as well as $\Del(X) \ll \min(X,A(X))^\ome$ for some $\ome \ge 0$. Then \eqref{1.mv-asymp} is satisfied whenever $\sig_0 \ge (2r+3) \ome /  \vrho$. In particular, when $\ome =\eps$, then \eqref{1.mv-asymp} holds for any $s > t_0$.  
\end{corollary}

One particular special case of Theorem~\ref{T:VMVT} and Corollary~\ref{cor-MV} relates to systems of polynomials $\bphi$ with non-vanishing Wronskian over the set $\calA = \N$ of integers. The following result is well known in the folklore and follows easily by standard arguments (see also recent work by Boyer and Robert \cite{BoyR}), but it is also a straightforward consequence of our results.

\begin{corollary}\label{cor-MV2}
	Let $\vphi_j(x)=c_jx^{k_j}$ for $1\le j \le r$ with $k_1 < \ldots < k_r$, and suppose that $\calA = \N$. Suppose further that $s> \frac12 k_{\mathrm{max}}(k_{\mathrm{max}}+1)$. Then \eqref{1.mv-asymp} is satisfied.
\end{corollary} 
This follows easily by observing that under the assumptions of the corollary, we have $A(X)=X$,  $E(X)=1$ and $Q_D=X$, and moreover Vinogradov's mean value theorem \cite[Theorem~1.1]{BDG}, \cite[Corollary 1.3]{W:NEC} together with standard arguments \cite[Theorem~5.2]{V:HL} imply that $\Delta(X) \ll X^\eps$ for $t_0 \ge \frac12 k_{\mathrm{max}}(k_{\mathrm{max}}+1)$, and $\vrho \ge 1/(k_{\mathrm{max}}(k_{\mathrm{max}}-1))$ with $Q_W=X$, respectively, so that any  $\sig>0$ is permitted.

\section{Approximation on the major arcs}\label{S3}

Recall that we aim for results that do not depend on the function $A(X)$ being well understood in terms of asymptotic formul{\ae}. In the case that there is an asymptotic formula, we write $A(X) = \Psi(X) + O(E(X))$ for a differentiable and increasing function $\Psi$ and an error term $E(X)$ which for the sake of simplicity we will assume to be the same as that of hypothesis \eqref{D}. We also write $\psi(x) = \Psi'(x)$. 

In the case when $A(X)$ is less well understood, we construct a piecewise differentiable function $\Psi^*$ as follows.  Suppose that the set $\calA^\circ$, viewed as a simple subset of the natural numbers, is given by $\calA^\circ = \{x_1, x_2, \ldots\}$, where we assume the $x_i$ to be listed in ascending order. Moreover, set $x_0=0$. We partition the positive real axis into intervals $I_n = (x_{n-1}, x_n]$ for $n \in \N$, and define the functions 
\begin{align}\label{3.functions}
	\psi^*_n(x) &= \frac{a_{x_n}}{|I_n|} \mathbbm{1}_{I_n}(x), & \psi^*(x) &= \sum_{n=1}^\infty \psi^*_n(x), & \Psi^*(X) = \int_0^X \psi^*(x) \d x. 
\end{align}
Clearly, $\Psi^*$ is monotonously increasing with the property that $\Psi^*(x_i) = A(x_i)$ for all $x_i \in \calA^{\circ}$. Moreover, the following lemma will help to control the error that arises upon replacing $A(X)$ by $\Psi^*(X)$. 
\begin{lemma}\label{L3.reg}
	Suppose that $\calA$ satisfies Condition~\eqref{D} with error term $E(X) = o(A(X))$. Then $a_n \ll E(n)$. 
\end{lemma}
\begin{proof}
	We may assume that $a_n > 0$, for otherwise the statement is trivial. Let $q \in \N$ be minimal with the property that there exist at least two distinct values $b \in \Z/q\Z$ with $\kap(q,b)>0$, and pick $b \not\equiv n \mmod q$ so as to maximise the value of $\kap(q,b)$. Thus, we trivially have $A(q,b;n)-A(q,b;n-1) = 0$. On the other hand, from Property~\eqref{D} we find that 
	\begin{align*}
		0&=A(q,b;n)-A(q,b;n-1) = \kap(q,b)[A(n)-A(n-1)] + O(E(n)) \\
		&= \kap(q,b) a_n + O(E(n)). 
	\end{align*}
	Since by construction $\kap(q,b)>0$, this implies the desired conclusion. 
\end{proof}
In particular, as a simple consequence of Lemma~\ref{L3.reg} we have $A(X)-\Psi^*(X) \ll E(X)$.

We now start our analysis, and here we will focus on the situation of Theorem~\ref{T:VMVT}, the proofs of Theorems~\ref{T:Waring} and~\ref{T:Waring2} being largely analogous (see also Section~\ref{S6} below). We write $\Psi$ for either an increasing and differentiable function approximating $A(X)$, or else $\Psi = \Psi^*$ as defined above, and write $\psi$ for the derivative of $\Psi$. Thus in either case we have 
\begin{align}\label{3.A-Psi-error}
	A(X)-\Psi(X) \ll E(X).
\end{align}
Recall also the definition of the exponential sum $f_{\bphi, \calA}(\balp;X)$ from Section \ref{subs: MVs}, and make the definitions
\begin{align*}
	S_{\calA}(q, \bb) = \sum_{\ell \mmod q} \kap(q,\ell) e \left( \frac{b_1 \vphi_1(\ell) + \ldots + b_r \vphi_r(\ell)}{q}\right)
\end{align*}
and
\begin{align*}
	v_{\calA}(\bbet;X)= \int_0^X \psi(\xi) e(\bet_1 \vphi_1(\xi) + \ldots + \bet_r \vphi_r(\xi)) \d \xi.
\end{align*}
When $\balp = \bb/q + \bbet$ for some $q \in \N$, $\bb \in \Z^r$ and $\bbet \in \R^r$, we introduce the shorthand notation
\begin{align*}
	f_{\bphi, \calA}^*(\balp; X) = S_\calA(q;\bb) v_{\calA}(\bbet;X).
\end{align*}
In all of these notations we drop the dependence on $\bphi$ and $X$ whenever there is no danger of confusion. 

\begin{lemma}\label{L3.approx}
	Assume Property~\eqref{D}, and let $\balp = \bb/q + \bbet$, where $q \le Q_D$. Then 
	$$
		f_\calA(\balp;X) - f_{\calA}^*(\balp;X) \ll \left(1 + \sum_{j=1}^r |\bet_j|X^{k_j}\right) \left(qE(X) + \frac{A(X)}{X}\right).
	$$
\end{lemma}
\begin{proof}
	Set 
	$$
		\Sig(t) = \sum_{x \in \calA(t)} a_x e \left(\frac{b_1 \vphi_1(x) + \ldots + b_r \vphi_r(x)}{q}\right)
	$$
	and define the discrete differencing operator $\partial$ via its action on a function $F$ by putting $\partial_t F(t)= F(t+1)-F(t)$. By Abel summation we can write 
	\begin{align}\label{3.5}
		f_{\calA}(\bb/q + \bbet) = \Sig(X) e(\bbet \cdot \bphi(X+1)) - \sum_{t \le X} \Sig(t) \partial_t e(\bbet \cdot \bphi(t)). 
	\end{align}
	We may sort the sums $\Sig(t)$ into congruence classes modulo $q$. Thus, we find from Condition~\eqref{D} and~\eqref{3.A-Psi-error} that 
	\begin{align}\label{3.6}
		\Sig(t) &= \sum_{\ell \mmod q} e\left(\frac{\bb \cdot \bphi(\ell)}{q}\right) A(q,\ell;t) + O(qE(t)) \nonumber\\
		&= S_\calA(q,\bb)\Psi(t) + O(qE(t)).
	\end{align}
	The mean value theorem now implies that 
	\begin{align}\label{3.7}
		\partial_t e(\bbet \cdot \bphi(t)) \ll \sum_{j=1}^r |\bet_j| t^{k_j-1}.
	\end{align}
	Inserting~\eqref{3.6} and~\eqref{3.7} into~\eqref{3.5}, we discern that 
	\begin{align}\label{3.8}
		f_{\calA}(\bb/q + \bbet)&= S_\calA(q,\bb) \left[\Psi(X) e(\bbet \cdot \bphi(X+1)) - \sum_{t \le X} \Psi(t) \partial_t e(\bbet \cdot \bphi(t))\right] \nonumber\\
		&\qquad + O\left(q E(X) \left(1 + \sum_{j=1}^r |\bet_j|X^{k_j}\right)\right) \nonumber\\
		&= S_\calA(q,\bb) \sum_{x \le X} \psi(x) e(\bbet \cdot \bphi(x)) + O\left(q E(X) \left(1 + \sum_{j=1}^r |\bet_j|X^{k_j}\right)\right),
	\end{align}
	where the last step follows by another application of Abel summation. 
	For $|x-\xi|<1$ we then have 
	\begin{align*}
		&|\psi(x)e(\bbet \cdot \bphi(x)) -  \psi(\xi)e(\bbet \cdot \bphi(\xi))| \\
		&\ll |\psi(x)-\psi(\xi)| + \psi(\xi)|e(\bbet \cdot \bphi(x)) - e(\bbet \cdot \bphi(\xi))|\\
		&\ll \psi(x)  + \psi(\xi)\left(1 + \sum_{j=1}^r |\bet_j|X^{k_j-1}\right)
	\end{align*}
	by the triangle inequality and the mean value theorem, and thus
	\begin{align}\label{3.9}
		\sum_{x \le X} \psi(x) e(\bbet \cdot \bphi(x)) - v_{\calA}(\bbet;X)
		&=\sum_{x \le X} \int_{x-1}^x \psi(x) e(\bbet \cdot \bphi(x)) - \psi(\xi)e(\bbet \cdot \bphi(\xi)) \d \xi \nonumber\\
		&\ll \sum_{x \le X}  \psi(x)  +  \left(1+ \sum_{j=1}^r |\bet_j|X^{k_j-1}\right)\int_0^X \psi(\xi) \d \xi  \nonumber\\
		&\ll \Psi(X) \left(1+ \sum_{j=1}^r |\bet_j|X^{k_j-1}\right).
	\end{align}
	Upon combining~\eqref{3.8} with~\eqref{3.9} and noting that $|S_\calA(q, \bb)| \le 1$ trivially as a consequence of~\eqref{2:kappa-add-1}, we thus obtain via~\eqref{3.A-Psi-error} that
	\begin{align*}
		f_{\calA}(\bb/q + \bbet) = S_\calA(q,\bb) v_\calA(\bbet;X) + O \left(\left(q E(X) + \frac{A(X)}{X}\right)\left(1 + \sum_{j=1}^r |\bet_j| X^{k_j}\right)\right)
	\end{align*}
	as claimed. 
\end{proof}

Now, for a parameter $Q$ we define the set $\frN(Q)$ of all $\balp \in \T^r$ having an approximation $\balp = \bb/q + \bbet$ with $1 \le b_j \le q \le Q$ and $|\bet_j| \le QX^{-k_j}$ for $1 \le j \le r$. Note that $\frN(Q)$ is a slight extension of our major arcs $\frM_X(Q)$. Consequently, the corresponding set of minor arcs $\frn(Q)=\T^r \setminus \frN(Q)$ is a strict subset of $\frm_X(Q)$. 

A standard computation shows that 
\begin{align*}
	\vol \frN(Q) \ll Q^{2r+1}X^{-K}.
\end{align*}
Upon combining this with Lemma~\ref{L3.approx}, we can compute the total error arising from replacing $f_\calA(\balp)$ by $f_{\calA}^*(\balp)$ on $\frN(Q)$, and obtain the bound 
\begin{align}\label{3.approxN}
	\int_{\frN(Q)} \left| |f_{\calA}(\balp)|^{2s} - |f^*_{\calA}(\balp)|^{2s}\right| \d \balp \ll A(X)^{2s} X^{-K}Q^{2r+3}\left(\frac{E(X)}{A(X)}+\frac{1}{X}\right).
\end{align}
Thus, we deduce the asymptotic formula
\begin{align}\label{3.decompV}
	I_{s, \bphi}(X; \calA) &= \int_{\frN(Q)} |f_{\calA}^*(\balp)|^{2s} \d \balp  \nonumber\\
	&\qquad + O \left(  \int_{\frm(Q)}|f_{\calA}(\balp)|^{2s} \d \balp + A(X)^{2s}X^{-K}Q^{2r+3}\left(\frac{E(X)}{A(X)}+\frac{1}{X}\right)\right).
\end{align}

We conclude the section by bounding the contribution coming from the minor arcs. 
\begin{proposition}\label{prop-minor}
	Suppose that Properties~\eqref{D}, \eqref{V} and~\eqref{W} are satisfied with parameters $Q_W$, $Q_D$, $\vrho$, $t_0$, $\Del$ and $L$, and define $Y$ as in~\eqref{1.Y-mv}. Under the assumptions of Theorem~\ref{T:VMVT} there exists a number $Q$ satisfying $\Del^{1/(\sig_0  \vrho)}L^{1/\vrho} \llcurly Q \llcurly Y$ for which	\begin{align*}
		I_{s, \bphi}(X; \calA) = \int_{\frN(Q)} |f_{\calA}^*(\balp)|^{2s} \d \balp + o(A(X)^{2s}X^{-K}). 
	\end{align*}
\end{proposition}

\begin{proof}
	Our starting point is~\eqref{3.decompV}. Write $\sig = 2s - 2t_0$, then by our assumptions we have $\sig > \sig_0$. Upon combining Properties~\eqref{V} and~\eqref{W}, we discern that 
	\begin{align*}
		\int_{\frm(Q)} |f_{\calA}(\balp)|^{2t_0+\sig} \d \balp & \ll \sup_{\balp \in \frm(Q)} |f_{\calA}(\balp)|^\sig \int_{\T^r} |f_{\calA}(\balp)|^{2t_0} \d\balp\\
		&\ll A(X)^{2t_0+\sig} X^{-K} (Q^{-\vrho}L)^\sig \Del. 
	\end{align*}
	This is acceptable whenever $\sig$ is large enough so that $Q \ggcurly \Del^{1/(\vrho\sig)}L^{1/\vrho}$. Our task is therefore to minimise $\sig$ by choosing $Q$ as large as possible under the given requirements. 
	
	From Conditions~\eqref{D} and~\eqref{W} we require that $Q \le \min\{Q_D, Q_W\}$. Moreover, in order to control the error arising from~\eqref{3.approxN}, we need $Q \llcurly \min\{(A/E)^{1/(2r+3)}, X^{1/(2r+3)}\}$. Taking $Y$ as in~\eqref{1.Y-mv}, we thus need to find a $Q$ satisfying 
	$$
		\Del^{1/(\vrho\sig)}L^{1/\vrho} \llcurly Q \llcurly Y. 
	$$
	The existence of such a $Q$ is guaranteed whenever $\sig > \sig_0$ where $\sig_0$ satisfies~\eqref{1.sig}. 
\end{proof}

\section{Analysis of the major arcs}\label{S4}
	It remains to understand the contribution arising from the major arcs. For a measurable set $\frB \subset \T^r$ define 
	$$
		I^*_{\calA}(\frB) = \int_{\frB} |f_{\calA}^*(\balp)|^{2s} \d \balp,
	$$
and write $I_{\calA}(Q)$ for $I^*_{\calA}(\frN(Q))$. 
We first observe that by a change of variables we have 
\begin{align}\label{4.rescaling}
	v_{\calA}(\bbet; X) &= \int_0^{\Psi(X)} e(\bbet \cdot \bphi (\Psi^{-1}(\zeta)) \d \zeta \nonumber \\
	&= \Psi(X) \int_0^1 e(\bbet \cdot \bphi (\Psi^{-1}(\Psi(X)\xi))) \d \xi.
\end{align}
The function $\Psi^{-1}(\Psi(X)\xi)$ maps $(0,1]$ to $(0,X]$. Put 
\begin{align}\label{4.betgam}
	\bgam = (\bet_j X^{k_j})_{1 \le j \le r}
\end{align}
and $z(\xi) = \Psi^{-1}(\Psi(X)\xi)/X$, so that $z$ is an endomorphism on $(0,1]$. Define next
\begin{align}\label{4.w-def}
	w_{\calA}(\bgam) = \int_0^1 e(\bgam \cdot \bphi(z(\xi))) \d \xi.
\end{align}
Now recall that we had assumed the $\bphi$ to be homogeneous, so that $\vphi_j(x)=c_j x^{k_j}$ for $1 \le j \le k$. It follows that~\eqref{4.rescaling} can be written as 
\begin{align}\label{4.vtow}
	v_\calA(\bbet;X) = \Psi(X) w_{\calA}(\bgam) \sim A(X)w_{\calA}(\bgam).
\end{align}
In particular, we deduce that 
\begin{align}\label{4.I(Q)}
	I_{\calA}(Q) = \Psi(X)^{2s}X^{-K} \frS_{\calA}(Q) \frJ_\calA(Q) \sim A(X)^{2s} X^{-K} \frS_{\calA}(Q) \frJ_\calA(Q),
\end{align}
where we defined the truncated singular series
\begin{align*}
	\frS_{\calA}(Q) = \sum_{q \le Q}\sum_{\substack{\bb \mmod q \\ (\bb, q)=1}}|S_\calA(q, \bb)|^{2s}
\end{align*}
and the truncated singular integral 
\begin{align*}
	\frJ_{\calA}(Q) = \int_{|\bgam| \le Q}|w_\calA(\bgam)|^{2s} \d \bgam.
\end{align*}

Our goal is to complete the singular series and singular integral by letting $Q$ tend to infinity, and thus to understand the limits 
\begin{align*}
    \frS_{\calA} = \lim_{Q \to \infty} \frS_{\calA}(Q)\qquad \text{and} \qquad \frJ_{\calA} = \lim_{Q \to \infty} \frJ_{\calA}(Q).
\end{align*}
Our strategy is to exploit to maximal effect the bounds provided by the Properties~\eqref{V} and~\eqref{W}. A similar approach has been pursued both in \cite{HB-Skoro} and in \cite{BW2}, but only in relation to the singular integral. 

We begin by bounding the exponential integral $w_\calA(\bgam)$ and the complete exponential sum $S_\calA(q, \bb)$. Here, we are using a modified version of an approach due to Browning and Heath-Brown \cite[Lemmata 8.2 and 8.3]{BHB}.

\begin{lemma}\label{L4.w-bd}
	Assume that Conditions~\eqref{D} and~\eqref{W} are satisfied with $Q_D$, $Q_W$, $L$ and $\vrho$, and assume that $Z$ is large enough so that 
	\begin{align}\label{4.Y-bd1}
		|\bgam| \le \min\left\{ Q_D(Z), Q_W(Z), \left(\frac{A(Z)}{E(Z)}\right)^{1/(1+\vrho)},Z^{1/2} \right\}.
	\end{align}
	Then 
	\begin{align*}
		w_\calA(\bgam) \ll (1+|\bgam|)^{-\vrho} L(Z).
	\end{align*}
\end{lemma}

\begin{proof}
	Without loss of generality, we assume that $|\bgam| > 1$, the lemma being trivially true otherwise. Set $Q=|\bgam|$ and note that~\eqref{4.Y-bd1} implies that $Q \le Z^{1/2}$. With this choice, the major arcs $\frM_Z(Q)$ are disjoint, and upon applying~\eqref{4.betgam} we see that $\bbet$ lies just on the edge of the major arcs with approximation $q=1$ and $\bb=\bzer$. Thus, by Lemma~\ref{L3.approx} and~\eqref{4.betgam}, and on recalling that $S_\calA(1,\boldsymbol{0})=1$ and $\kap(1)=1$, we find that 
	\begin{align*}
		v_\calA(\bbet; Z) = f_{\calA}(\bbet; Z) + O\left( \left(E(Z) + \frac{A(Z)}{Z}\right)|\bgam|\right).
	\end{align*} 
	At the same time, by continuity the bound of Property~\eqref{W} continues to apply on the boundary of the major arcs. Thus, upon recalling~\eqref{4.vtow}, we discern that
	\begin{align*}
		w_{\calA}(\bgam) &\ll A(Z)^{-1}f_{\calA}(\bbet; Z) + |\bgam|\left(\frac{E(Z)}{A(Z)} +  \frac{1}{Z}\right)\\&\ll  Q^{-\vrho} L(Z) + |\bgam|\left(\frac{E(Z)}{A(Z)} +  \frac{1}{Z}\right).
	\end{align*}
	 At this point, \eqref{4.Y-bd1} ensures that the first term dominates, whence the claim follows. 
\end{proof}

In many applications, including those discussed in Corollary~\ref{cor-W}, the requirement~\eqref{4.Y-bd1} may be replaced by demanding that $Z$ be a suitably large (but finite) power of $|\bgam|$. In  such circumstances, we have $L(Z) = L(|\bgam|^c) \ll |\bgam|^\eps$ since by assumption $L(Z) \ll Z^\eps$.

\begin{lemma}\label{L4.q-bd}
	Assume that Conditions~\eqref{D} and~\eqref{W} are satisfied with $Q_D$, $Q_W$,  $L$ and $\vrho$, and assume that $Z$ is large enough so that 
	\begin{align}\label{4.q-bd}
		q \le \min\left\{ Q_D(Z), Q_W(Z), \left(\frac{A(Z)}{E(Z)}\right)^{1/(1+\vrho)},Z^{1/2} \right\}.
		\end{align}	
	Then for any $\bb \in (\Z/q\Z)^r$ with $(q,\bb)=1$ we have
	\begin{align*}
		S_\calA(q, \bb) \ll q^{-\vrho}L(Z).
	\end{align*}
\end{lemma}
\begin{proof}
	Again, we assume without loss of generality that $q>1$, and define $Q = q$. Fix $Z$ satisfying~\eqref{4.q-bd}, then the major arcs $\frM_Z(Q)$ are disjoint. Moreover, the point $\bb/q$ lies in this set of major arcs, and it is in fact best approximated by itself. Since~\eqref{4.vtow} implies that $v_{\calA}(\boldsymbol{0};Z) = A(Z) w_\calA(\boldsymbol{0}) = A(Z)$, we discern from Lemma~\ref{L3.approx} in a manner similar to before that 
	\begin{align*}
		S_\calA(q, \bb)  \ll A(Z)^{-1} f_{\calA}(\bb/q;Z) + q\left(\frac{E(Z)}{A(Z)} + \frac{1}{Z}\right).
	\end{align*}
	On the other hand, $Q$ has been chosen such that the point $\bb/q$ lies only  marginally on the major arcs in the $q$-aspect, and is consequently not contained in $\frM_Z(Q/2)$. Thus, again by using the minor arc bound of~\eqref{W} we deduce that 
	\begin{align*}
		S_\calA(q, \bb)  \ll Q^{-\vrho} L(Z) + Q\left(\frac{E(Z)}{A(Z)} + \frac{1}{Z}\right),
	\end{align*}	
	and when $Z$ satisfies~\eqref{4.q-bd} the first term can be seen to dominate. This establishes the claim of the lemma. 
\end{proof}

As in the previous case, we remark that in many applications the hypothesis~\eqref{4.q-bd} can be replaced by demanding that $Z$ be a suitably large power of $q$, so that $L(Z) \ll q^\eps$.

Define next 
\begin{align}\label{4.scrS}
	\scrS_{\calA}(Q) = \sum_{q \le Q} \sum_{\substack{\bb=1 \\ (\bb,q)=1}}^q |S_{\calA}(q, \bb)|^{2t_0}  
\end{align}
and 
\begin{align}\label{4.scrJ}
	\scrJ_{\calA}(Q) = \int_{|\bgam| \le Q}|w_\calA(\bgam)|^{2t_0} \d \bgam.
\end{align}
It is not hard to bound both of these quantities from below. 
\begin{lemma}\label{L4.lower}
	We have $\scrS_{\calA}(Q) \gg 1$ and $\scrJ_{\calA}(Q) \gg 1$ uniformly in $Q$. 
\end{lemma}
\begin{proof}
	The first statement follows trivially upon noting that $\scrS_{\calA}(Q)$ is a sum of non-negative numbers, with the term corresponding to $q=1$ being equal to $1$. 
	
	For the second statement we use a similar idea. Clearly, we have $w_{\calA}(\bzer)=1$. Moreover, we have the directional derivatives 
	$$
		\nabla_{\bgam} e(\bgam \cdot \bphi(z(\xi))) = 2 \pi i \bphi(z(\xi)).
	$$ 
	It follows by an application of the mean value theorem that \begin{align*}
		|w_{\calA}(\bzer) - w_{\calA}(\bgam)| \le \int_0^1 |1-e(\bgam \cdot \bphi(z(\xi)))| \d \xi \le 2 \pi r M|\bgam|,
	\end{align*}	
	where 
	$$
	M= \max_{1 \le j \le r} \sup_{\xi \in [0,1]} |\vphi_j(\xi)|.
	$$
	Thus we have
	$$
		w_{\calA}(\bgam)\ge 1/2 \quad \text{for} \quad |\bgam| \le \frac{1}{4 M \pi r},
	$$
	and we conclude that 
	$$
		\int_{|\bgam| \le Q}|w_\calA(\bgam)|^{2t_0} \d \bgam \ge (4 M \pi r)^{-r} 2^{-2t_0} \gg 1
	$$
	as desired. 
\end{proof}
We can leverage Lemma~\ref{L4.lower} in order to obtain upper bounds for the quantities $\scrS_{\calA}(Q)$ and $\scrJ_{\calA}(Q)$. 

\begin{lemma}\label{L4.upper}
	Assume that Properties~\eqref{D} and~\eqref{V} are satisfied with $Q_D$, $t_0$ and $\Del$, and suppose that $Z$ is large enough so that 
	\begin{align}\label{4.Y-bd3}
		Q \le \min \left\{Q_D(Z), \left(\frac{A(Z) \Del(Z)}{E(Z)}\right)^{1/(2r+3)}, Z^{1/(2r+3)} \right\}.
	\end{align}
	Then we have $\scrS_\calA(Q) \ll \Del(Z)$ and $\scrJ_\calA(Q) \ll \Del(Z)$. 
\end{lemma}

\begin{proof}
	From~\eqref{4.betgam}, \eqref{4.vtow} and~\eqref{4.scrJ} we see that 
	\begin{align*}
		\scrJ_{\calA}(Q) = A(Z)^{-2t_0} Z^K \int_{\substack{|\bet_j| \le QZ^{-k_j} \\ (1 \le j \le r)}} |v_{\calA}(\bbet; Z)|^{2t_0} \d \bbet.
	\end{align*}
	Combining this with~\eqref{4.scrS} and Lemma~\ref{L3.approx}, we find that 
	\begin{align*}
		\scrJ_{\calA}(Q) \scrS_{\calA}(Q) &= A(Z)^{-2t_0} Z^K \int_{\frM_Z(Q)} |f^*_{\calA}(\balp; Z)|^{2t_0}\ \d \balp\\
		&= 	A(Z)^{-2t_0} Z^K \int_{\frM_Z(Q)} |f_{\calA}(\balp; Z)|^{2t_0} \d \balp + O\left(Q^{2r+3}\left(\frac{E(Z)}{A(Z)} + \frac{1}{Z}\right)\right).
	\end{align*}
	Since the integrand is non-negative, we may extend the range of integration to the full unit torus. Thus, upon applying Property~\eqref{V}, we discern that 
	\begin{align*}
		\scrJ_{\calA}(Q) \scrS_{\calA}(Q) \ll \Del(Z) + Q^{2r+3} \left(\frac{E(Z)}{A(Z)} + \frac{1}{Z}\right), 
	\end{align*}
	and~\eqref{4.Y-bd3} ensures that the first term dominates. The conclusion of the lemma now follows upon combining this last bound with Lemma~\ref{L4.lower}
\end{proof}

We are now prepared to show that both the singular series and the singular integral can be extended to infinity and converge absolutely. 

Recall the definition of $Y$ from (\ref{1.Y-mv}). For any $Q > 0$ take $Z_Q$ so that $Q = Y(Z_Q)$. This is always possible since $Y(Z)$ is a function that increases to infinity, and which we can assume to be continuous with $Y(0)=0$. Take then $\sig>\sig_0$ where $\sig_0$ is as in~\eqref{1.sig}. With this choice, Lemmata~\ref{L4.w-bd} and~\ref{L4.upper} are applicable and yield 
\begin{align}\label{4.J-dyad}
	\frJ(2Q) - \frJ(Q)& \ll \sup_{|\bgam| \ge Q} |w_\calA(\bgam)|^{\sig} \scrJ(2Q) \nonumber \\
	&\ll Q^{-\sig \vrho}L(Z_Q)^\sig \Del(Z_Q).
\end{align}
Suppose that $\sig = \sig_0+\sig_1$ for some $\sig_1>0$. Recall also that the hypotheses of Theorem~\ref{T:VMVT} stipulate that  $L(Z) \ll Y(Z)^{\vrho-\del}$ for some $\del>0$. Thus, upon applying the definition~\eqref{1.sig}, we can rewrite the expression on the right-hand side of~\eqref{4.J-dyad} as
\begin{align*}
	Q^{-\sig \vrho}L(Z_Q)^\sig \Del(Z_Q)  \ll (Q^{-\vrho}L(Z_Q))^{\sig_1-\eps} \ll Q^{-\del_1}
\end{align*}
for some $\del_1>0$. 
Thus, we discern that 
\begin{align}\label{4.J-conv}
	\frJ_{\calA} - \frJ_{\calA}(Q) = \sum_{j=1}^\infty \left( \frJ_{\calA}(2^jQ) - \frJ_{\calA}(2^{j-1}Q) \right) \ll \sum_{j=0}^\infty(2^jQ)^{-\del_1} \ll Q^{-\del_1}.
\end{align}

In a similar manner, we see from Lemmata~\ref{L4.upper} and~\ref{L4.q-bd} that 
\begin{align}\label{4.S-dyad}
	\frS_{\calA}(2Q)-\frS_{\calA}(Q) &\ll \sup_{q \ge Q} \sup_{\substack{\bb \mmod q \\ (q, \bb)=1}} \left|S_\calA(q,\bb)\right|^{\sig} \scrS_{\calA}(2Q) \nonumber\\
	&\ll Q^{-\sig\vrho} L(Z_Q)^\sig \Del(Z_Q) \ll Q^{-\del_1},
\end{align}
and consequently 
\begin{align}\label{4.S-conv}
	\frS_{\calA} - \frS_{\calA}(Q) \ll Q^{-\del_1}.
\end{align}

We summarise our results up to this point.
\begin{proposition}\label{prop-asymp}
	Suppose that Properties~\eqref{D}, \eqref{V} and~\eqref{W} are satisfied with parameters $Q_W$, $Q_D$, $\vrho$, $t_0$, $\Del$ and $L$, and define $Y$ as in~\eqref{1.Y-mv}. Under the assumptions of Theorem~\ref{T:VMVT} there exists a number $Q$ satisfying $\Del^{1/(\sig_0  \vrho)}L^{1/\vrho} \llcurly Q \llcurly Y$ for which	
	\begin{align*}
		I_{s, \bphi}(X; \calA) = A(X)^{2s}X^{-K} \frS_{\calA}\frJ_{\calA} + o(A(X)^{2s}X^{-K}). 
	\end{align*}
\end{proposition}

\begin{proof}
	This is immediate from Proposition~\ref{prop-minor} upon combining \eqref{4.I(Q)} with \eqref{4.S-conv} and \eqref{4.J-conv}.
\end{proof}

\section{Interpretation of the major arcs}\label{S5}
The next step is to interpret the singular series $\frS_{\calA}$ and singular integral $\frJ_{\calA}$. Here, the singular integral carries the weights $(a_n)$, making its treatment more challenging, while the analysis of the singular series is mostly standard (see in particular \cite[Chapter~2.6]{V:HL} or \cite[Chapter~5]{dav}). However, we choose to give details regardless even in this setting, partly in the interest of a self-contained presentation, but also to address in the appropriate detail the deviations from the usual arguments caused by our setting over thin sets.  

We begin with the singular series. Throughout, we will now assume that Condition~\eqref{C} is satisfied for the set $\calA$. The first main result here states that $\frS_{\calA}$ factors into a product of local factors. For any natural number $q$ write
$$
	B_{\calA}(q) = \sum_{\substack{\bb=1 \\ (\bb,q)=1}}^q |S_{\calA}(q, \bb)|^{2s}.
$$
The following lemma is essentially standard (see e.g. \cite[Lemma~2.10]{V:HL}), but due to our more general setting we include a proof for the sake of completeness.
\begin{lemma}
	Assume Property \eqref{C}. Then the function $B_{\calA}(q)$ is multiplicative. 
\end{lemma}
\begin{proof}
	The statement follows if we can show that for any coprime pair $(q,q')$ and any $\bb, \bb'$ we have 
	\begin{align}\label{5.S-mult}
		S_{\calA}(q; \bb) S_{\calA}(q';\bb') = S_{\calA}(qq';q\bb'+q'\bb).
	\end{align}
	
	The right-hand side of~\eqref{5.S-mult} can be written as
	\begin{align*}
		S_{\calA}(qq'; q\bb' + q'\bb) &= \sum_{q'x+qy \mmod {qq'}} \kap(qq',q'x+qy) e \left(\frac{\sum_{j=1}^r (qb'_j + q'b_j)\vphi_j(qy+q'x) }{qq'}\right).
	\end{align*}
	Using the polynomial structure of the the functions $\vphi_j$, we note next that \begin{align*}
		\vphi_j(qy+q'x) \equiv \vphi_j(qy) + \vphi_j(q'x) \pmod{qq'},
	\end{align*}
	since any cross term will contain a factor $qq'$, and by assumption the $\vphi_j$ have no constant term. Using this identity together with our above considerations, we find that 
	\begin{align*}
		&\sum_{q'x+qy \mmod {qq'}} \kap(qq',q'x+qy) e \left(\frac{\sum_{j=1}^r (qb'_j + q'b_j)\vphi_j(qy+q'x) }{qq'}\right) \\
		&= \sum_{q'x+qy \mmod {qq'}} \kap(qq',q'x+qy) e \left(\frac{\sum_{j=1}^r (qb'_j + q'b_j)(\vphi_j(qy) +\vphi_j(q'x)) }{qq'}\right)\\
		&=\sum_{x \mmod q} \kap(q,q'x) e \left(\frac{\sum_{j=1}^r b_j\vphi_j(q'x)}{q}\right) \sum_{y \mmod q} \kap(q',qy) e \left(\frac{\sum_{j=1}^r b'_j\vphi_j(qy) }{q'}\right),
	\end{align*}
	where we used \eqref{C} in the last step. After a final change of variables, we see that this last product is equal to the expression on the left-hand side of~\eqref{5.S-mult}. 
\end{proof}

Thus, upon formally setting 
$$
	\chi_p(\calA) = \sum_{h =0}^{\infty} B_{\calA}(p^h)
$$
we can write the singular series as the formal Euler product
$$
	\frS_{\calA} = \prod_p \chi_p(\calA).
$$
Moreover, since we have already established the absolute convergence of $\frS_{\calA}$, it follows that the Euler product has the same property. 

It remains to interpret the factors $\chi_p(\calA)$. To this end, define 
\begin{align*}
	\Gam_{\calA}(q) = q^{-r}\sum_{\bb \mmod q}\left|S_{\calA}(q; \bb)\right|^{2s},
\end{align*}
so that $\Gam_{\calA}(q)$ counts the solutions $\bx, \by \in (\Z/q\Z)^s$ satisfying 
$$
	\vphi_j(x_1)+ \ldots + \vphi_j(x_s) \equiv \vphi_j(y_1)+\ldots + \vphi(y_s) \pmod q \qquad (1 \le j \le k),
$$
but where each solution $(\bx,\by)$ is counted with weight $\prod_{i=1}^s \kap(q,x_i)\kap(q,y_i)$. In this notation we have the following.
\begin{lemma}\label{L5.Gam}
	The factor $\chi_p(\calA)$ is given by 
	\begin{align*}
		\chi_p(\calA) = \lim_{h\to \infty}p^{rh}\Gam_{\calA}(p^h).
	\end{align*}
\end{lemma}

\begin{proof}	
	By taking out common divisors, we have 
	\begin{align*}
		\Gam_{\calA}(p^h) = p^{-rh} \sum_{j=0}^h \sum_{\substack{\bc = 1 \\ (p, \bc)=1}}^{p^{j}} |S_{\calA}(p^h; p^{h-j}\bc)|^{2s}.
	\end{align*} 
	The inner exponential sum can be simplified by observing that
	\begin{align*}
		S_{\calA}(p^h; p^{h-j}\bc) &= \sum_{x \in \Z/p^h\Z} \kap(p^h,x)e \left(\frac{c_1 \vphi_1(x) + \cdots + c_r \vphi_r(x)}{p^j}\right) \\
		&= \sum_{x \mmod{p^j}} e \left(\frac{c_1 \vphi_1(x) + \cdots + c_r \vphi_r(x)}{p^j}\right) \sum_{\substack{y \mmod{p^h} \\ y \equiv x \mmod {p^j}}}\kap(p^h,y)\\
		& = S_{\calA}(p^j; \bc),
	\end{align*}
	where in the last step we used \eqref{2:kappa-add}. 
	Consequently, we obtain
	\begin{align*}
		\chi_p(\calA) &= \lim_{h\to \infty} \sum_{j =0}^h \sum_{\substack{\bc = 1 \\ (p, \bc)=1}}^{p^j}|S_{\calA}(p^j; \bc)|^{2s}= \lim_{h\to \infty}  p^{rh} \Gam_{\calA}(p^h)
	\end{align*}
	as desired.
\end{proof}

To conclude the discussion of the singular series, we recall that $\Gam_{\calA}(q)$ describes the weighted count of solutions to the system of $r$ equations~\eqref{1.mv-eq}, viewed as congruences modulo $q$. Since the weights $\kap(q,b)$ are normalised so that $\sum_{b \mmod q}\kap(q,b)=1$, but the number of choices is constrained by $r$ equations, the anticipated value of $\Gam_{\calA}(q)$ is $q^{-r}$. Thus, the quantity $q^r\Gam_{\calA}(q)$ describes the weighted density of solutions to~\eqref{1.mv-eq} in $\Z/q\Z$, and upon restricting to prime powers and taking the limit $h \to \infty$ we conclude that $\chi_p(\calA)$ describes the weighted density of solutions of~\eqref{1.mv-eq} over the embedding of $\calA$ into the field of $p$-adic numbers $\Q_p$. \bigskip

Having conducted the analysis of the singular series to a satisfactory endpoint, we next direct our attention to the singular integral. 
Recall that 
$$
	\frJ_{\calA} = \int_{\R^r} |w_{\calA}(\bgam)|^{2s} \d \bgam,
$$
where $w_{\calA}(\bgam)$ is given by~\eqref{4.w-def}. Define the $\calA$-normalised measure of level $X$ on $\R_+$ via 
$$
	\d\sig_{\calA}(z) = \frac{X}{\Psi(X)} \psi(Xz)\d z.
$$ 

For every $X$, we have $\sig_{\calA}([0,1]) = 1$, so this is a probability measure on $[0,1]$ that preserves the structure of the set $\calA(X)$. 
Note that since the density distribution of $\calA(X)$ is not homogeneous and does depend on the cutoff point $X$, the scale $X$ cannot be entirely dispensed with. For easier notation, we will also abbreviate 
$$
	\prod_{j=1}^s \d \sig_{\calA}(\xi_i) = \d \sig_{\calA}(\bxi).
$$ 

By a change of variables, we can then write $w_{\calA}(\bgam)$ in the shape
$$
	w_{\calA}(\bgam) = \int_0^1 e(\bgam \cdot \bphi(z)) \d \sig_{\calA}(z). 
$$

Next, we describe a heuristic argument regarding the interpretation of $\frJ_{\calA}$. Define 
\begin{align}\label{5.V_A}
	V_{\calA}(\bh) 
	&=
	\int_{[0,1]^{2s}} \prod_{j=1}^r \mathbbm{1}_{[ \vphi_j(\xi_1) + \ldots + \vphi_j(\xi_s) - \vphi_j(\zeta_1) - \ldots - \vphi_j(\zet_s) = h_j] } \d \sig_{\calA}(\bxi)\d \sig_{\calA}(\bzet)\nonumber\\
	&= \vol_{\calA}\left\{\bxi, \bzet \in [0,1]^s: \sum_{i=1}^s(\bphi(\xi_i) - \bphi(\zeta_i) )= \bh \right\},
\end{align}
where the notation $\vol_{\calA}$ indicates that the volume is to be interpreted in terms of the measure $\sig_{\calA}$. With this notation, we see that $V_{\calA}(\bzer)$ describes the $\sig_{\calA}$-weighted volume of the set of $(\bxi, \bzet) \in [0,1]^s$ satisfying~\eqref{1.mv-eq}. We can then compute the inverse Fourier transform 
\begin{align*}
	\calF^{-1} V_{\calA} (\bgam) &= \int_{\R^r} V_{\calA}(\bh) e(\bh \cdot \bgam) \d \bh \\
	& =\int_{[0,1]^{2s}} e(\bgam \cdot (\bphi(\xi_1) + \ldots + \bphi(\xi_s) - \bphi(\zeta_1) - \ldots - \bphi(\zeta_s) ))  \d \sig_{\calA}(\bxi)\d \sig_{\calA}(\bzet)\\
	&= |w_\calA(\bgam)|^{2s},
\end{align*}
so that upon taking the forward Fourier transform and formally evaluating at $\bzer$ we retrieve the identity  
\begin{align*}
	V_{\calA}(\bzer) = \int_{\R^r} |w_{\calA}(\bgam)|^{2s} \d \bgam = \frJ_{\calA}.
\end{align*}
In other words, whenever these considerations are valid, the singular integral $\frJ_{\calA}$ has an interpretation in terms of the $\sig_{\calA}$-weighted volume of the solutions of~\eqref{1.mv-eq} inside $[0,1]^{2s}$. 

    Unfortunately, the fact that Fourier transforms are well-defined only in an almost-everywhere sense means that the reasoning described here needs to be taken with a grain of salt unless we can show that the Fourier transform converges pointwise at zero. One sufficient criterion for the pointwise convergence of Fourier series is that the function in question be differentiable. Thus, the argument sketched above is valid provided that the collection of polynomials $\bphi$ is such that the variety defined by~\eqref{1.mv-eq} is smooth inside the cube $(0,1)^{2s}$. 

In order to avoid convergence issues, we follow a more robust approach due to Schmidt \cite[Section 11]{schmidt82}.
When $T \ge 1$, define
\begin{align*}
	h_T(y) = \begin{cases} T(1-T|y|), & \text{when } |y| \le T^{-1},\\ 0, & \text{otherwise}, \end{cases}
\end{align*}
and recall that
\begin{align*}
	h_T(y) = \int_{-\infty}^{\infty}  e(\bet y)  \left( \frac{\sin(\pi \bet/T)}{\pi \bet/T}\right)^2 \d \bet,
\end{align*}
where the integral converges absolutely. Define the kernel
\begin{align*}
	K_T(\bgam) = \prod_{j=1}^{r} \left( \frac{\sin(\pi \gam_j/T)}{\pi \gam_j /T}\right)^2
\end{align*}
and set 
\begin{align*}
	W_T = \int_{\R^r} | w_{\calA}(\bgam)|^{2s} K_T(\bgam) \d \bgam,
\end{align*}
so that 
\begin{align}\label{5.W_T}
	W_T = \int_{[0,1]^s}  \prod_{j=1}^r h_T( \vphi_j(\xi_1) + \ldots + \vphi_j(\xi_s) - \vphi_j(\zet_1) - \ldots - \vphi_j(\zet_s)) \d \sig_{\calA}(\bxi) \d \sig_{\calA}(\bzet).
\end{align}
Our first step is to show that $W_T$ is an approximation to $\frJ_{\calA}$. In order to bound the difference 
\begin{align}\label{5.Schmidt-diff}
	W_T - \frJ_{\calA} = \int_{\R^r} |w_{\calA}(\bgam)|^{2s}(K_T(\bgam) - 1) \d \bgam,
\end{align}
it is convenient to consider two domains separately. Write $U_1=[-\sqrt T, \sqrt T]^r$, and set 
$U_2 = \R^r \setminus U_1$. The power series expansion of $K_T$ gives
\begin{align*}
	0\le 1-K_T(\bgam) \ll \min \left\{ 1,  \sum_{j=1}^{r} (|\gam_j|/T)^2\right\},
\end{align*}
and we discern that the domain  $U_1$ contributes at most
\begin{align*}
	\left(\sup_{\bgam \in U_1}|1 - K_T(\bgam)|\right)\int_{\R^r} |w_{\calA}(\bgam) |^{2s} \d \bgam \ll \sup_{\bgam \in U_1}|1 - K_T(\bgam)| \frJ_{\calA} \ll T^{-1},
\end{align*}
where we used our result from the previous section that $\frJ_{\calA} \ll 1$. 

Meanwhile, the contribution from $U_2$ is bounded above by
\begin{align*}
	\sum_{i=1}^{\infty} |\frJ_{\calA}(2^i \sqrt T) - \frJ_{\calA}(2^{i-1} \sqrt T)| \ll  \sum_{i=1}^{\infty}(2^i \sqrt T)^{-\del} \ll T^{-\del/2},
\end{align*}
for some positive number $\del$  with $\del<1$, where again we took advantage of 
our earlier findings. Thus we infer from~\eqref{5.Schmidt-diff} that
\begin{equation}\label{5.Schmidt1}
	|W_T - \frJ_{\calA}| \ll T^{-\del/2}
\end{equation}
for all $T \ge 1$, and hence $W_T$ does indeed converge to $\frJ_{\calA}$, as 
claimed.

We now turn to the second ingredient in Schmidt's approach, a geometric argument to show that $W_T \gg 1$ uniformly in $T$. Here, we need to refine his approach of \cite[Lemma~2]{schmidt82} to better accommodate our $\calA$-weighted measure. 

In order to correctly interpret the results, recall that the $\calA$-weighted measure is not usually equidistributed, and that its distribution depends in particular on the scale $X$. Thus, while the measure is normalised with respect to the unit cube, the volume of a subset of the unit cube will depend on how the subset correlates with the mass distribution of the measure. Since the latter depends on $X$, the implicit constant in the estimate $W_T \gg 1$ will also depend on $X$, but it will, crucially, be uniform in $T$. In order to illustrate this phenomenon, we therefore give the following more quantitative version of \cite[Lemma~2]{schmidt82}.
	
	\begin{lemma}\label{L5.Schmidt2}
		Suppose that the function $\psi$ has at most finitely many discontinuities on the interval $[0,X]$. Then we have $W_T \gg V_{\calA}(\bzer)$ uniformly for all sufficiently large $T$, where $V_\calA(\bzer)$ is as defined in~\eqref{5.V_A}.
	\end{lemma}
	\begin{proof}
		Define the $r$-tuple of functions 
		$$
			\Phi_j(\bx) = \sum_{i=1}^s\left(\vphi_j(x_i) - \vphi_j(x_{s+i})\right) \qquad (1 \le j \le r)
		$$
		and the manifold
		$$
			\calM = \left\{\bx \in [0,1]^{2s}: \Phi_j(\bx)=0 \qquad (1 \le j \le r)\right\},
		$$
		so that $V_{\calA}(\bzer)$ denotes the $\calA$-weighted $(2s-r)$-dimensional volume of $\calM$. Let $\pi:\R^{2s} \to \R^{2s-r}$ denote the map that is defined by the projection of a point $\bxi \in \R^{2s}$ onto its first $2s-r$ coordinates. Since $\calM$ is algebraic of finite degree, we may assume without loss of generality that a generic fibre $\pi^{-1}(\bzet)$ for $\bzet \in \pi(\calM)$ is a finite set with a bounded number of elements, for should this not be the case, it could easily be arranged by means of a change of variables.  Consequently, in this notation, we have $\vol_{\calA}(\pi(\calM))\gg V_{\calA}(\bzer)$.
		
		For simpler notation, write 
		$$
			\psi_1(\xi) = \frac{\psi(X\xi)}{A(X)} \quad \text{and} \quad \widehat\psi(\bxi) = \prod_{j=1}^{2s} \psi_1(\xi_j), 
		$$
		and partition $\calM$ into a disjoint set of submanifolds $(\calM_j)_{j \in \calJ}$ such that $\widehat\psi$ is continuous on each piece $\calM_j$. 
		Next fix some suitable constant $\tau >0$.  For each $j \in \calJ$ we can find a submanifold $\calM'_j \subseteq \calM_j$ with $\dim \calM'_j=2s-r$ which lies a positive distance $\del$ from the boundary of the unit cube and satisfies $\psi_1(\xi) \ge \tau$ for all components of $\bxi$ whenever $\bxi \in \calM'_j$. Thus, for each $j \in \calJ$ there is a family $\calJ'(j)$ of indices and a pairwise disjoint family of subsets $\bigcup_{j' \in \calJ'(j)} \calM'_{j,j'}\subseteq\calM'_j $ such that each piece $\calM'_{j,j'}$ can be parametrised by the first $2s-r$ coordinates. We abbreviate the notation by setting $\calI = \{(j,j'): j' \in \calJ'(j), j \in \calJ\}$, so that $\calM'_{j,j'} = \calM'_i$. 
		Set now $ \bxi= (\bzet, \btet) \in \R^{2s-r} \times \R^{r}$. Then, for each $i \in \calI$ we can fix an open set $\scrO_i \subseteq [0,1]^{2s-r}$ and a continuous map $f: \scrO_i \rightarrow \R^{r}$ such that $(\bzet, f(\bzet)) \in \calM'_i$ for all $\bzet \in \scrO_i$.
		
		The set
		\begin{align*}
			\scrU_i(\del) = \left\{(\bzet, \btet): \bzet \in \scrO_i, \left|f(\bzet)-\btet\right| < \del  \right\}
		\end{align*}
		still lies in the unit cube.  
		Thus we can compute 
		\begin{align*}
			\vol_{\calA} (\scrU_i(\del)) &= \int_{\scrU_i(\del)} \widehat\psi(\bxi)\d \bxi \\
			&= \int_{\scrO_i}  \left(\int_{f(\bzet) + [-\del, \del]^r}  \left(\prod_{j=1}^{r}\psi_1(\tet_j)\right) \d \btet   \right) \left(\prod_{j=1}^{2s-r}  \psi_1(\zet_j) \right)\d \bzet.
		\end{align*}		
		Since the density function $\psi_1$ is Lipschitz continuous on $\calM_i'$ with some constant $\lam$, we have that $|\psi_1(\tet_j) - \psi_1( f_j(\bzet))| \le \lam |\tet_j - f_j(\bzet)|$ for $1 \le j \le r$. Thus, whenever $\lam\del<\tau/2$, we find that the inner integral is $\gg (\tau\del)^r$, and consequently we obtain the bound
		\begin{align}\label{5.volS}
			\vol_{\calA} (\scrU_i(\del)) \gg_\tau \del^r \vol_{\calA}(\scrO_i). 
		\end{align}

		Furthermore, since $\calM'_i \subseteq \calM$ one has $\Phi_j(\bxi)=0$ for $1 \le j \le r$, and therefore again upon exploiting Lipschitz continuity we deduce that $\left|\Phi_{ j}(\bzet, \btet)\right| \le  (2T)^{-1}$ for $(\bzet, \btet) \in \scrU_i(cT^{-1})$ for some suitable $c>0$. This implies that 
		\begin{align}\label{5.Lipschitz-lowerbd}
			\prod_{j}h_T(\Phi_{j}(\bxi)) \gg T^r \qquad \text{	for all $\bxi \in \scrU_i(cT^{-1})$}.
		\end{align} 
		
		We can now insert the bounds~\eqref{5.Lipschitz-lowerbd} and~\eqref{5.volS} into~\eqref{5.W_T}, and obtain that 
		\begin{align*}
	  		W_T& \gg \sum_{i \in \calI}\int_{\scrU_i(cT^{-1})}  \left(\prod_{j=1}^r h_T(\Phi_{j}(\bxi))\right) \d \sig_{\calA}(\bxi) \\
            &\gg \sum_{i \in \calI} T^r \vol_{\calA}(\scrU_i(cT^{-1})) \\
            &\gg \sum_{i \in \calI}\vol_{\calA}\scrO_i
	  	\end{align*}
	  	whenever $T$ is large enough in terms of the constants $\del$, $c$, $\lam$ and $\tau$.
	  	
		Set $\scrO = \bigcup_{i \in \calI} \scrO_i$. In this notation, by our construction we have clearly $\scrO \subseteq \pi(\calM)$. Moreover, by fixing $\tau$ and $\delta$ small enough, our initial manipulations allow us to choose the set $\scrO$ such as to very nearly exhaust the set $\pi(\calM)$. Thus, we can choose $\scrO$ large enough so that $\vol_{\calA}\scrO \ge \mu V_{\calA}(\bzer)$ for some small constant $\mu$. This completes the proof of the lemma. 
	\end{proof}
	To conclude our arguments, it suffices to combine the conclusion of Lemma~\ref{L5.Schmidt2} with~\eqref{5.Schmidt1} in order to deduce that $\frJ_{\calA} \gg V_{\calA}(\bzer)$. Thus the singular integral is seen to be positive as soon as the rescaled solution set has a positive $\calA$-weighted volume in the real unit cube.

\section{Proofs of Theorems~\ref{T:Waring} and~\ref{T:Waring2}}\label{S6}

As practitioners of the circle method will be well aware, the proofs of Theorems~\ref{T:Waring} and~\ref{T:Waring2} differ only slightly from that of Theorem~\ref{T:VMVT}. However, for the sake of completeness, we give an overview of the changes needed to derive these results. 

We focus first on Theorem~\ref{T:Waring}, and observe that in the setting at hand, it is natural to make the choice $X=n^{1/k}$. Moreover, in this setting we have $r=1$. 
The arguments of Section~\ref{S3} apply now identically, up until the point where~\eqref{3.decompV} is replaced by 
\begin{align*}
	R_{s;k}(n;\calA) &= \int_{\frN(Q)} f_{\calA}^*(\alp)^s  e(-n \alp) \d \alp + O \left(  \int_{\frm(Q)}|f_{\calA}(\alp)|^{s} \d \alp + A^s n^{-1}Q^{5}\left(\frac{E}{A}+\frac{1}{X}\right)\right).
\end{align*}
Consequently, Proposition~\ref{prop-minor} should then be replaced by the following. 
\begin{proposition}\label{prop-W}
	Suppose that Hypotheses~\eqref{D}, \eqref{V1} and~\eqref{W1} are satisfied with parameters $Q_W$, $Q_D$, $\vrho$, $t_0$, $\Del$ and $L$, and define $Y$ as in~\eqref{1.Y}. Under the assumptions of Theorem~\ref{T:Waring}, we have 
	\begin{align*}
		R_{s;k}(n;\calA) = \int_{\frN(Q)} f_{\calA}^*(\alp)^s e(-\alp n) \d \alp + o(A(n^{1/k})^sn^{-1}).
	\end{align*}
\end{proposition}
The arguments of Section~\ref{S4} now apply again identically. Finally, in Section~\ref{S5} we replace our definition of $B_\calA(q)$ by 
$$
	B_\calA(q;n) = \sum_{\substack{b=1 \\ (b,q)=1}}^q S_{\calA}(q, b)^{s} e(-bn/q),
$$
and find entirely analogously to before that under Condition~\eqref{C} the singular series $\frS_{\calA}(n)$ has a product representation of the form 
$$
	\frS_{\calA}(n) = \prod_p \chi_p(n;\calA),
$$
where 
$$
	\chi_p(n;\calA) = \sum_{h=0}^{\infty} B_\calA(p^h;n).
$$
Lemma~\ref{L5.Gam} holds by the same argument with the amended counting function 
\begin{align*}
	\Gam_{\calA}(q) &= \sum_{b \mmod q} S_\calA(q,b)^s e(-b/q)\\
	&=\sum_{\bx \mmod q}\sum_{\by \mmod q} \mathbbm{1}_{[x_1^k+\ldots+x_s^k \equiv n \mmod q]} \kap(q,x_1) \cdots \kap(q,x_s).
\end{align*}
Thus, the singular series has an interpretation in terms of $p$-adic solution densities of the underlying equation as before.

In a similar way, the singular integral is now given by 
$$
	\frJ_{\calA}(n) = \int_{\R} w_{\calA}(\gam)^s e(-\gam) \d \gam,
$$
and has an interpretation as the volume with respect to the measure $\sig_{\calA}$ of the set of $\xi_1, \ldots, \xi_s \in [0,1]$ satisfying $\xi_1^k + \ldots + \xi_s^k = 1$. Since this variety is smooth, the argument surrounding Equation~\eqref{5.V_A} is sufficient in this case. 
\medskip

We turn now to the proof of Theorem~\ref{T:Waring2}. Just as before, put $X=n^{1/k}$. Define the exponential sum 
$$
	g(\alp;X) = \sum_{x \le X} e(\alp x^k),
$$
so that 
$$
	R_{s,u;k}(n;\calA) = \int_{\T} f_{\calA}(\alp;X)^s g(\alp;X)^u e(-n\alp) \d \alp.
$$

From \cite[Theorem~2.1(b)]{BP2} we have the bound 
\begin{align}\label{6.weyl}
	\sup_{\alp \in \frm_X(Q)} |g(\alp;X)| \ll X Q^{-\vrho_0}  \log X,
\end{align}
where $\vrho_0$ is as defined in~\eqref{0.rho}. Thus, $g(\alp;X)$ satisfies Condition~\eqref{W1} with $\vrho=\vrho_0$ and $L(X)=\log X$. Define next the complete exponential sum 
$$
	S(q,a) = q^{-1} \sum_{x=1}^q e(ax^k/q)
$$
and the exponential integral 
$$
	v(\bet; X)= \int_0^X e(\bet x^k) \d x,
$$
and for $\alp = a/q+\bet$ write $g^*(\alp) = S(q,a)v(\bet;X)$. In this situation, \cite[Theorem~4.1]{V:HL} states that 
$$
	|g(\alp)-g^*(\alp)| \ll q^{1/2+\eps}(1 + X^k|\bet|)^{1/2}.
$$
Suppose now that $\alp \in \frN(Q)$. Then, upon combining this last bound with Lemma~\ref{L3.approx}, we obtain 
\begin{align*}
	|f_{\calA}(\alp)^s g(\alp)^u - f_{\calA}^*(\alp)^s g^*(\alp)^u| \ll A(X)^sX^{u}Q^2\left(\frac{E(X)}{A(X)}+\frac{1}{X}\right),
\end{align*}
and hence the analogue of~\eqref{3.approxN} is
\begin{align}\label{6.approxN}
	\int_{\frN(Q)} |f_{\calA}(\alp)^s g(\alp)^u - f_{\calA}^*(\alp)^s g^*(\alp)^u| \d \alp \ll A(X)^sX^{u-k}Q^5\left(\frac{E(X)}{A(X)}+\frac{1}{X}\right).
\end{align}
Our analogue of Proposition~\ref{prop-minor} then takes the following shape.
\begin{proposition}\label{prop-W2}
	Suppose that Hypotheses~\eqref{D} and~\eqref{V1} are satisfied with parameters $Q_D$, $t_0$ and  $\Del$, and define $Y$ as in~\eqref{1.Y2}. Under the assumptions of Theorem~\ref{T:Waring2}, we have 
	\begin{align*}
		R_{s,u; k}(n; \calA) = \int_{\frN(Q)} f_{\calA}^*(\alp)^s g^*(\alp)^u e(-\alp n) \d \alp + o(A(X)^sX^{u-k}).
	\end{align*}
\end{proposition}
\begin{proof}
	From Hypothesis~\eqref{V1} together with~\eqref{6.weyl}, we deduce that 
	\begin{align*}
		\int_{\frm(Q)} |f_{\calA}(\alp)^s g(\alp)^u| \d \alp &\ll \sup_{\frm(Q)} |g(\alp)|^u \int_{\T} |f_{\calA}(\alp)|^s \d \alp \\
		&\ll (X \log X)^u Q^{-u\vrho} A(X)^{s}X^{-k}\Del.
	\end{align*}
	To control this error, we require that $Q^{u\vrho} \ggcurly \Del (\log X)^u$. On the other hand, just like in the proof of Proposition~\ref{prop-W}, we also need $Q \llcurly Y$ in order to make sure that Hypothesis~\eqref{D} is applicable and the error in~\eqref{6.approxN} is under control. Together, this implies the statement just as in the earlier proposition. 
\end{proof}

In the analysis of the major arc contribution, we again set $\gam = X^k \bet$ and  
$$
	w(\gam;X) = \int_0^1 e(\gam \xi^k) \d \xi,
$$ 
so that $v(\bet;X)= X w(\gam)$. 
Set now 
$$
	\frS^*_{\calA} = \sum_{q \le Q} \sum_{\substack{b \mmod q \\ (b,q)=1}} S_{\calA}(q,b)^s S(q,b)^u e(-bn/q)
$$
and 
$$
	\frJ_{\calA}^*(Q) = \int_{|\gam| \le Q} w_\calA(\gam)^s w(\gam)^u e(-\gam ) \d \gam.
$$
We could bound $S(q,b)$ and $w(\gam)$ by the strong bounds presented in \cite[Chapter~4]{V:HL}, but we prefer to imitate the analysis of Section~\ref{S4} in order to illustrate how the bounds thus obtained, albeit weaker than those given in Vaughan's monograph, are still fully sufficient for our purposes. 

Observe that Lemmata~\ref{L4.w-bd} and~\ref{L4.q-bd}, respectively, applied to $\calA = \N$, show the following. 

\begin{lemma}\label{L6.weyl}
	Assume that $Z \ge |\gam|^2$. Then 
	\begin{align*}
		w(\gam) \ll (1+|\gam|)^{-\vrho} \log(Z).
	\end{align*}
	
	Similarly, when $Z \ge q^2$, then for any $b \in \Z/q\Z$ with $(q,b)=1$ we have 
	\begin{align*}
		S(q,b) \ll q^{-\vrho} \log(Z).
	\end{align*}
\end{lemma}
Meanwhile, Lemma~\ref{L4.upper} applies identically. For any $Q>0$ pick now $Z_Q$ such that $Q=Y(Z_Q)$, where $Y$ is as in~\eqref{1.Y2}. 
Thus, upon combining Lemma~\ref{L4.upper} with Lemma~\ref{L6.weyl}, we see that in the place of~\eqref{4.J-dyad}, we obtain 
\begin{align*}
	\frJ_\calA^*(2Q) - \frJ^*_\calA(Q)& \ll \sup_{|\bgam| \ge Q} |w(\gam)|^{u} \scrJ(2Q)\ll Q^{-u \vrho}\log(Z_Q)^u \Del(Z_Q).
\end{align*}
Similarly, the analogue of~\eqref{4.S-dyad} is given by 
\begin{align*}
	\frS_\calA^*(2Q) - \frS^*_\calA(Q)& \ll \sup_{q \ge Q}  \max_{b \in (\Z / q \Z)^*}|S(q,b)|^{u} \scrS(2Q)\ll Q^{-u \vrho}\log(Z_Q)^u \Del(Z_Q).
\end{align*}
Recall now that in the hypotheses of Theorem~\ref{T:Waring2} we demand that $u>\sig_0$, where $\sig_0$ is given via~\eqref{1.sig2}. Thus, the same arguments as in the end of Section~\ref{S4} show, \emph{mutatis mutandis}, that both the singular series and the singular integral may be extended to infinity. 

In the discussion of the singular series, we define now 
$$
	B_{\calA,u}(q;n) = \sum_{\substack{b=1 \\ (b,q)=1}}^q S_{\calA}(q, b)^{s} S(q, b)^{u}e(-bn/q).
$$
Assuming Condition~\eqref{C}, the singular series $\frS_{\calA}^*(n)$ has a product representation 
$$
\frS^*_{\calA}(n) = \prod_p \chi^*_p(n)
$$
as before, where now the factors $\chi^*_p$ are defined in terms of $B_{\calA,u}(q;n)$ rather than $B_{\calA}(q;n)$. Again, a very slight modification of the arguments of Lemma~\ref{L5.Gam} show that 
$$
	\chi^*_p(n) = \lim_{h \to \infty} p^{h} \Gam_{\calA,u}(p^h),
$$
where this time we have the function 
\begin{align*}
	\Gam_{\calA,u}(q) &= \sum_{b \mmod q} S_\calA(q,b)^s S(q,b)^u e(-b/q)\\
	&=q^{-u} \sum_{\bx \mmod q}\sum_{\by \mmod q} \mathbbm{1}_{[x_1^k+\ldots+x_s^k+y_1^k+\ldots+y_u^k \equiv n \mmod q]} \kap(q,x_1) \cdots \kap(q,x_s),
\end{align*}
and we obtain an interpretation of the singular series in terms of $p$-adic solution densities of the underlying equation just as in the other cases.

In a similar manner, the same deliberations as in the previous section lead us to a singular integral of the shape 
$$
	\frJ^*_{\calA,u}(n) = \int_{\R} w_{\calA}(\gam)^s w(\gam)^u e(-\gam) \d \gam.
$$
For $\bxi \in [0,1]^s$ and $\bzet \in [0,1]^u$ define the $(s+u)$-dimensional measure $\sig^{\dagger}_{\calA}$ by putting 
$$
	\d \sig^{\dagger}_{\calA}(\bxi, \bzet) = \d \sig_{\calA}(\bxi) \wedge \d \bzet.
$$
In this notation, the same arguments as before show that $\frJ^*_{\calA,u}(n)$ has an interpretation as the volume with respect to $\sig^{\dagger}_{\calA}$ of the set of $\bxi \in [0,1]^s, \bzet \in [0,1]^u$ for which 
$$
	\xi_1^k + \ldots + \xi_s^k + \zet_1^k + \ldots + \zet_u^k = 1.
$$ 

\section{Proofs of the teaser theorems}
It remains to show how the results teased in the introduction can be obtained from our general theorems. 
We begin with the proof of Theorem~\ref{T0:Akshat}. Define $\Psi=\Psi^*$ as in \eqref{3.functions}, counting every element $x \in \calA$ with weight $1$, and recall the definitions of the logarithmic lower density $\lam(\calA)$ and of $\vrho_0(k)$ from~\eqref{0.def-lam} and~\eqref{0.rho}, respectively.
	
Consider first the situation described in \eqref{T0:Akshat-VMVT}. Since the mean value $I_{t,k}(X;\calA)$ describes the number of solutions $x_1, \ldots, x_{2t} \in \calA(X)$ of the equation 
$$
	x_1^k+\ldots+x_t^k = x_{t+1}^k+\ldots + x_{2t}^k,
$$
it then follows immediately from \cite[Corollary~14.8]{W:NEC} (see also \cite[Theorem~2.1]{BP2}) that 
$$
	I_{t,k}(X;\calA) \le I_{t,k}(X;\N) \ll X^{2t-k+\eps}
$$ 
for all $t \ge \tfrac12 \vrho_0(k+1)^{-1}$. Consequently, Hypothesis~\eqref{V1} is satisfied with $t_0 = \tfrac12 \vrho_0(k+1)^{-1}$ and
$$
	\Delta(X) \ll A(X)^{\eps}(X/A(X))^{\vrho_0(k+1)^{-1}} \ll A(X)^{{\vrho_0(k+1)^{-1}}(\lam^{-1}-1)+\eps}.
$$  
Clearly, for $\lam^{-1} < 1+\vrho_0(k+1)\vrho_0(k)/5$ we have $\vrho_0(k+1)^{-1}(\lam^{-1}-1) < \vrho_0(k)/5$ and hence $\Del(X) \ll A(X)^{\vrho_0/5-\del}$ for some $\del>0$ (assuming that $\eps$ had been taken small enough).

Consider now the situation described in part \eqref{T0:Akshat-convex} of the theorem. As mentioned in the discussion of the result, the main ingredient here is a recent result by Mudgal \cite{mudgal-convex} concerning higher energies of $k$-convex sets. An ordered set $\calA = \{x_1, x_2, \ldots\} \subseteq \N$ is called $1$-convex if the sequence of its consecutive differences is strictly increasing, that is, $x_{n+1}-x_n < x_{n+2}-x_{n+1}$ for all $n$. One then continues inductively by defining a set $\calA$ to be $k$-convex if the set of its consecutive differences is a $(k-1)$-convex set. It is easy to see that the set of $k$-th powers of a set $\calA$ satisfying condition~\eqref{T0:Akshat-cond1} of Theorem~\ref{T0:Akshat} will be $(k-1)$-convex.
	
It thus follows immediately from \cite[Theorem~1.6]{mudgal-convex} that Hypothesis~\eqref{V1} is satisfied with $\Delta(X) \ll (X/A(X))^k  A(X)^{\vrho_0/10} \ll A(X)^{k(\lam^{-1}-1) + \vrho_0/10}$ and $t_0 = \widetilde C 2^k \log(10k/\vrho_0)$ for some suitable constant $\widetilde C$. Clearly, for $\lam^{-1} < 1-\vrho_0(k)/(10k)$ we have $k(\lam^{-1}-1) < \vrho_0/10$ and hence $\Del(X) \ll A(X)^{\vrho_0/5-\del}$ for some $\del>0$, so that Corollary~\ref{cor-W2} is applicable. We note finally that in the range $k \ge 6$ we have $\vrho_0^{-1}=(k-1)(k-2) + 2 \lfloor \sqrt{2k} \rfloor$, so that for these values of $k$ the bound becomes $s> 10 \widetilde C 2^k \log k$. Clearly, for $k \in \{2,3,4,5\}$ a similar conclusion may be obtained by, if necessary, modifying the constant $\widetilde C$. 
	
In either case, the conclusion of Theorem~\ref{T0:Akshat} now follows from Corollary~\ref{cor-W2} upon noting that condition~\eqref{D} was already assumed as one of the hypotheses of the theorem.

Before establishing our remaining teasers, we pause to show that \eqref{D} automatically implies \eqref{C} whenever $\kap(q,b)$ is of the shape 
\begin{align}\label{3.kap}
	\kap(q,b) = \begin{cases}(\# \calA_q)^{-1} & \text{if } b \in \calA_q, \\ 0 & \text{else,} \end{cases}
\end{align}
for a set of residue classes $\calA_q \subseteq \Z/q\Z$. 
\begin{lemma}\label{L3.kap-mult}
	Suppose that $\calA$ satisfies Property~\eqref{D} with $\kap$ of the shape given by \eqref{3.kap} for all $q$ and $b$. Then \eqref{C} holds. 
\end{lemma}
\begin{proof}
	Observe first that for $\kap$ given by \eqref{3.kap}, the relation \eqref{2:kappa-add} is equivalent to 
	\begin{align}\label{3.3}
		(\#A_q)^{-1} = (\#\calA_{qq'})^{-1} \# \{c \in \calA_{qq'}: c \equiv b \mmod q\}.
	\end{align}
	Suppose now that $(q,q')=1$. As a consequence of the Chinese remainder theorem, every $c \in \Z/qq'\Z$ has a unique representation of the form $c = \ell'q+\ell q'$ with $\ell \in \Z/q\Z$ and $\ell' \in \Z/q'Z$. Moreover, the condition that $c \equiv b \mmod q$ implies that $\ell \in \Z/q\Z$ is fixed. Consequently, we have 
	\begin{align}\label{3.4}
		\# \{ c \in \calA_{qq'}: c \equiv b \mmod q   \} = \# \{ \ell' \in \Z/q'\Z: \ell'q+ \ell q' \in \calA_{qq'} \},
	\end{align} 
	where $\ell$ is the unique element in $\Z/q\Z$ for which $\ell q' \equiv b \mmod q$. Denote the set on the right-hand side of~\eqref{3.4} by $\calU_{q,q'}(\ell)$. By the definition of $\calA_{qq'}$, the relation $\ell' \in \calU_{q,q'}(\ell)$ implies that there exists some $n \equiv \ell' q + \ell q' \mmod{qq'}$ for which $a_n >0$. Since clearly $n \equiv \ell'q \mmod{q'}$, we discern that $\ell' \in \calU_{q,q'}(\ell)$ precisely when $\ell'q \in \calA_{q'}$, and since $(q,q')=1$, we are forced to conclude that $\#\calU_{q,q'}(\ell)=\#\calA_{q'}$. Together with~\eqref{3.4} and~\eqref{3.3}, this completes the proof. 
\end{proof}	
	
We now turn to the proof of Theorem \ref{T0:Kirsti}.
As in the previous example, we define $\Psi=\Psi^*$ as in \eqref{3.functions}, and count every element of $\calE_p$ with weight $1$. By definition, our ellipsephic sets are equidistributed over a set of admissible residue classes modulo the base prime $p$ up to an error of at most $1$, so that $\kappa(p,b) = r^{-1}$ for $b \in \calD_p$, where $r=\#\calD_p$ denotes the number of permissible residue classes modulo $p$, and $\kappa(p,b)=0$ otherwise. When $q=p^h$, we have $r^h$ permissible residue classes $\mmod q$, and these are again equidistributed by construction, so that $A(p^h,b;X) = A(X)/r^h + O(1)$ for all admissible values $b$. Meanwhile, when $(q,p)=1$, then congruence conditions modulo $q$ and modulo $p$ are independent, and we have $\kappa(q,b)=q^{-1}$ for all $b \in \Z/q\Z$. Finally, by the Chinese remainder theorem the congruence conditions combine, so that for general $q = q_1 q_2$ with $(q_1, p)=1$ and $q_2 = p^h$, we obtain equidistribution with $\kap(q,b)=(q_1r^h)^{-1}$ for all admissible residue classes $b \in \Z/q_1q_2\Z$. Consequently, the ellipsephic set $\calE_p$ satisfies Hypothesis~\eqref{D} with $E(X)=1$ and $Q_D(X)=X$. Moreover, since the function $\kap$ is of the shape~\eqref{3.kap}, we also have Hypothesis~\eqref{C}.

We see from \cite[Theorem 1.2]{biggs2}, in combination with an argument akin to that of \cite[Chapter~7.1]{V:HL}, that for $t_0=m k(k+1)/2$, we have
\begin{align*}
    I_{t,k}(X;\calE_p) \ll X^{-k+\eps} A(X)^{2t} (X/A(X)^{m})^{k(k+1)/2}\qquad \text{for all $t \ge t_0$},
\end{align*}
which implies that the ellipsephic set $\calE_p$ satisfies Hypothesis~\eqref{V1} with 
$$
    \Del(X)=X^{\eps}(X/A(X)^{m})^{k(k+1)/2}.
$$
Suppose that the condition \eqref{0.ellips-cond} is satisfied. Then, recalling (\ref{0.def-lam}), we obtain
$$
\Del(X)\ll A(X)^{\eps +\frac{k(k+1)}{2}(\lambda^{-1}-m)} \ll A(X)^{\ome}
$$
for some $\ome < \vrho_0/5$, given a sufficiently small value of $\eps$. 
Consequently, we have met the hypotheses of Corollary \ref{cor-W2}, and Theorem \ref{T0:Kirsti} follows.

Finally, we turn to the case $\calA = \calP$ of primes, so that $A(X)=\pi(X)$ is the prime counting function. In this setting, as a result of the prime number theorem we may take $\Psi(\xi) = \li(\xi)$ for $\xi \ge \tau$ for some $\tau>2$, where as usual 
\begin{align}\label{7.li}
	\li(X) = \int_2^X \frac{\d t}{\log t} = \frac{X}{\log X} + O \left(\frac{X}{(\log X)^2}\right).
\end{align} 
For $0 \le \xi \le \tau$ we put $\Psi(\xi)=\li(\tau) \xi/\tau$, so that $\psi(\xi)=\li(\tau)/\tau$. 
We also recall the Siegel--Walfisz theorem (see e.g. \cite[Corollary~11.21]{MV:MNT}), which states that for any constant $B$ and for any $q \le (\log X)^B$ and any $b$ coprime to $q$ one has
\begin{align*}
	A(q,b;X)= \frac{\pi(X)}{\vphi(q)} + O(X e^{-\gamma \sqrt {\log X}}),
\end{align*}
where the constant $\gamma$ depends on $B$, and $\vphi$ denotes Euler's totient function. Thus in particular condition~\eqref{D} is true with parameters $Q_D = (\log X)^B$ and $E(X) = X e^{-\gamma \sqrt {\log X}}$. Since $\kap(q,b)=\vphi(q)^{-1}$ for $(q,b)=1$ and $\kap(q,b)=0$, it follows from Lemma~\ref{L3.kap-mult} that we also have Condition \eqref{C}.
	
We observe next that by work of Wooley \cite[Corollary 14.8]{W:NEC} we have 
$$
	I_{t,k}(X;\calP) \le I_{t,k}(X; \N)\ll X^{2t-k}
$$
whenever $t \ge k(k-1)/2 + \lfloor\sqrt{2k+1} \rfloor$. It follows that condition~\eqref{V1} holds in this setting with $t_0 = k(k-1)/2 + \lfloor \sqrt{2k+1} \rfloor$ and $\Delta(X) = (\log X)^{2t}$. At the same time, it follows from work by Harman \cite{Harman}, specifically in the formulation given in \cite[Lemma~3.3]{KT:survey}, that Hypothesis~\eqref{W1} is satisfied with $Q_W(X) = X^{1/2}$, $\vrho = 4^{1-k}$ and $L(X)=(\log X)^c$ for some suitable constant $c = c(k)$. 
	
With these parameters, we find that the quantity $Y(X)$ defined in~\eqref{1.Y} is given by $Y(X) = (\log X)^B$. Moreover, we are free to choose $B$ arbitrarily large, so we may assume that $B> 4^k(t_0+c(k))$. With this choice, we certainly have $L(X)\ll Y(X)^{\vrho - \del}$, and~\eqref{1.sig} is seen to be satisfied with some $\sig_0 \le 1/2$. Thus, all the hypotheses of Theorem~\ref{T:Waring} are satisfied, and we conclude that we have an asymptotic formula of the shape 
\begin{align}\label{7.WG-asymp}
	R_{s;k}(n; \calP) = \li(n^{1/k})^s n^{-1} (\frS_{\calP}(n) \frJ_{\calP}(n) + o(1))
\end{align}
whenever $s \ge 2t_0+1$. 
	
It remains to show that the main term agrees with the main term stated in Theorem~\ref{T0:Goldbach}. It is not hard to see that the singular series $\frS_{\calP}(n)$ of our analysis in Section~\ref{S6} is the same that occurs in \cite[Satz~11]{Hua}. For the singular integral, we note that with our approximation $\Psi(\xi) = \li(\xi)$ in the range $\xi > \tau$ we have $\psi(\xi)=1/\log \xi$ by the fundamental theorem of calculus, and hence 
\begin{align*}
	\d \sig_{\calP}(z)=\begin{cases}\begin{displaystyle}\frac{X}{\li X} \frac{\d z}{\log (Xz)}\end{displaystyle} & z \ge \tau/X, \\
	\begin{displaystyle}\frac{X \li (\tau)}{\li( X) \tau} \d z \end{displaystyle}& 0 \le z <\tau/X.
	\end{cases} 
\end{align*}
Thus, we obtain
\begin{align}\label{7.w_P}
	w_{\calP}(\gam) &= \frac{X}{\li X} \left(\frac{\li(\tau)}{\tau} \int_0^{\tau/X} e(\gam z^k) \d z + \int_{\tau/X}^1 e(\gam z^k) \frac{\d z}{\log (Xz)}\right) \nonumber \\
	&\sim \frac{\log X}{X} \int_\tau^X \frac{e(\gam \xi^k)}{\log \xi} \d \xi + O(1),
\end{align}
where we used~\eqref{7.li} together with substitution. At this point we note by the triangle inequality and on observing that the function $1/\log x$ is decreasing that 
\begin{align*}
	\int_\tau^X \frac{e(\gam \xi^k)}{\log \xi} \d \xi - \frac{1}{\log X} \int_0^X e(\gam \xi^k)\d \xi 
	&\ll \int_\tau^X \left|\frac{1}{\log \xi}-\frac{1}{\log X}\right|\d \xi + \frac{\tau}{\log X}\\
	&\ll \li X - \frac{X}{\log X} +O(1) \ll\frac{X}{(\log X)^2}.
\end{align*}
Using this bound in~\eqref{7.w_P}, we find that $w_{\calP}(\gam)  \sim w_{\N}(\gam)$ for any $\gam$, and thus 
\begin{align*}
	\frJ_{\calP}(n) &\sim \frJ_{\N}(n) =\frac{\Gam(1+1/k)^s}{\Gam(s/k)}
\end{align*}
by a standard result on the integer Waring problem (see e.g. \cite[Theorem~4.1]{dav}). Upon combining this with~\eqref{7.WG-asymp} and recalling~\eqref{7.li} and the functional equation of the Gamma function we discern that 
\begin{align*}
	R_{s;k}(n; \calP) \sim \frac{\Gam(1+1/k)^s}{\Gam(s/k)} \frS_{\calP}(n) \frac{n^{s/k-1}}{(\log n^{1/k})^{s} }= \frac{\Gam(1/k)^s}{\Gam(s/k)} \frS_{\calP}(n) \frac{n^{s/k-1}}{(\log n)^s},
\end{align*}
which finally establishes the theorem.

\end{document}